\documentclass[12pt]{article}
\usepackage{amsthm}
\usepackage{amsmath,amsfonts,amssymb,rotating,colordvi}
\usepackage{color}
\usepackage[symbol]{footmisc}
\usepackage[unicode=true,
 bookmarks=true,bookmarksnumbered=true,bookmarksopen=true,bookmarksopenlevel=2,
 breaklinks=false,pdfborder={0 0 1},backref=false,colorlinks=true]
 {hyperref}

\usepackage{tikz}

\usetikzlibrary{calc}
\usetikzlibrary{shapes.geometric}

\tikzset{
  c/.style={every coordinate/.try}
}

\newcommand{\red}{\mathrm{red}}

\usetikzlibrary{arrows,shapes,positioning}
\usetikzlibrary{decorations.markings}
\tikzstyle arrowstyle=[scale=1]
\tikzstyle directed=[postaction={decorate,decoration={markings,mark=at position 0.6 with {\arrow[arrowstyle]{stealth};}}}]
\tikzstyle reverse directed=[postaction={decorate,decoration={markings,mark=at position 0.4 with {\arrowreversed[arrowstyle]{stealth};}}}]
\tikzstyle dot=[style={circle,inner sep=1pt,fill}]

 \newtheorem{thm}{Theorem}[section]
\newtheorem{deff}[thm]{Definition}
\newtheorem{lem}[thm]{Lemma}
\newtheorem{prop}[thm]{Proposition}
\newtheorem{obs}[thm]{Observation}

\newtheorem{coro}[thm]{Corollary}
\newtheorem{conj}[thm]{Conjecture}
\newtheorem{remark}[thm]{Remark}
\numberwithin{equation}{section}

\newdimen\Squaresize \Squaresize=11pt
\newdimen\Thickness \Thickness=0.7pt
\def\Square#1{\hbox{\vrule width \Thickness
   \vbox to \Squaresize{\hrule height \Thickness\vss
    \hbox to \Squaresize{\hss#1\hss}
   \vss\hrule height\Thickness}
\unskip\vrule width \Thickness} \kern-\Thickness}

\def\Vsquare#1{\vbox{\Square{$#1$}}\kern-\Thickness}

\def\moins{\raise 1pt\hbox{{$\scriptstyle -$}}}

\begin{document}

\begin{center}
\textbf{\large{On the 12-representability of induced subgraphs of a grid graph}}\textbf{ }
\par\end{center}

\begin{center}

Joanna~N.~Chen\footnote[2]{E-Mail: \texttt{joannachen@tjut.edu.cn}} and
Sergey~Kitaev\footnote[3]{E-Mail: \texttt{sergey.kitaev@cis.strath.ac.uk}}\footnote[1]{The corresponding author} \par \bigskip

\begin{small}
\textsuperscript{1}College of Science, Tianjin University of Technology, Tianjin 300384, P.R. China \par
\textsuperscript{2}Department of Computer and Information Sciences, \\ University of Strathclyde, Glasgow, UK \par
\bigskip
\end{small}
\end{center}

\begin{abstract}

The notion of a 12-representable graph was introduced by Jones et al.\ in \cite{JKPR15}. This notion generalizes the notions of the much studied permutation graphs and co-interval graphs. It is known that any 12-representable graph is a comparability graph, and also that a tree is 12-representable if and only if it is a double caterpillar. Moreover, Jones et al.\ initiated the study of 12-representability of  induced subgraphs of a grid graph, and asked whether it is possible to characterize such graphs. This question in \cite{JKPR15} is meant to be about  induced subgraphs of a grid graph that consist of squares, which we call square grid graphs. However, an induced subgraph in a grid graph does not have to contain entire squares, and we call such graphs line grid graphs.

In this paper we answer the question of Jones et al.\ by providing a complete characterization of $12$-representable square grid graphs in terms of forbidden induced subgraphs. Moreover, we conjecture such a characterization for the line grid graphs and give a number of results towards solving this challenging conjecture. Our results are a major step in the direction of characterization of all 12-representable graphs since beyond our characterization, we also discuss relations between graph labelings and 12-representability, one of the key open questions in the area.\\


\end{abstract}

\section{Introduction}

Let $\mathbb{P}=\{1,2,\ldots\}$ and $\mathbb{P}^*$ be the set of all words over $\mathbb{P}$.
Given a word $w=w_1 w_2 \cdots w_n \in \mathbb{P}$, denote by $A(w)$ the set of integers occurring
in $w$. For example, $A(315353)=\{1,3,5\}$. For $B \subset A(w)$, let $w_B$ be the word obtained from $w$
by removing all the letters in $A(w)-B$. For example, if $B=\{2,3\}$ and $w=12315251$ then $w_B=232$. Let $\red(w)$ be the word that results from
$w$ by replacing each occurrence of the $i$-th smallest letter that occurs in $w$ by $i$. For example, $\red(3729)=2314$.

Let $u=u_1u_2\cdots u_m \in \mathbb{P}^*$ with $\red(u)=u$. Then, we say that  a word $w=w_1 w_2 \cdots w_n \in \mathbb{P}^*$ contains an occurrence of the {\emph pattern} $u$ if
there exist integers $1 \leq i_1<i_2<\cdots <i_m \leq n$ such that $\red(w_{i_1} w_{i_2}\cdots w_{i_m} )=u$. For example, the word $624635$ contains two occurrences of the pattern $4231$, namely the subsequences $6452$ and $6352$. The pattern $u$ is {\em consecutive}, if in each of its occurrences $i_{t+1}-i_{t}=1$ for all $1\leq t\leq m-1$.

Given a labeled graph $G=(V,E)$ and a pattern $u$, we say that $G$ is \emph{$u$-pattern representable} if there is a word $w \in \mathbb{P}^*$ such that $A(w)=V$, and for all $x,y \in V$, $xy \notin E$ if and only if $w_{\{x,y\}}$ contains an occurrence of $u$.
In such a situation, we say that $w$ \emph{$u$-pattern represents} $G$,  and $w$ is  called a \emph{$u$-pattern-representant } of $G$. An unlabeled graph $H$ is $u$-pattern representable if it admits a labeling resulting in a $u$-pattern  representable labeled graph $H'$. We say that $H'$ \emph{realizes $u$-pattern representability} of $H$.

Requiring from $u$ to be a consecutive pattern, we obtain the notion of a {\em $u$-representable graph} introduced in \cite{JKPR15}. In this case, similarly to the above, we can define $u$-representability and $u$-representants, or just {\em representants} if $u$ is clear from the context. The class of $u$-representable graphs generalizes the much studied class of word-representable graphs \cite{K17,KL15}, which is precisely $11$-representable graphs. It was shown in \cite{K17} that if a consecutive pattern $u$ is of length at least 3 then any graph can be $u$-represented. Also, note that a word avoids the pattern 12 if and only if it avoids the consecutive pattern 12, and thus the notion of a 12-pattern representable graph is equivalent to that of a 12-representable graph, which is the subject of interest in this paper.

Jones et al.~\cite{JKPR15} showed that the notion of a 12-representable graph generalizes the notions of the much studied {\em permutation graphs} (e.g.\ see~\cite{GRR13,LNP2009} and references therein) and {\em co-interval graphs} (e.g.\ see \cite{N2011,Z2009}). Also, Jones et al.~\cite{JKPR15}  showed that any 12-representable graph is a {\em comparability graph} (i.e.\ such a graph admits a {\em transitive orientation}), and also that a tree is 12-representable if and only if it is a {\em double caterpillar} (see \cite{JKPR15} for definition). More relevant to this paper, Jones et al.~\cite{JKPR15} initiated the study of 12-representability of  induced subgraphs of a {\em grid graph}, and asked whether it is possible to characterize such graphs. Examples of a grid graph and some of its possible induced subgraphs are given in Figure~\ref{grid-graphs}  which also appears in \cite{JKPR15}. Jones et al. \cite{JKPR15} showed that corner and skew ladder graphs, and thus ladder graphs, are 12-representable, while any graph with an induced cycle of size at least 5 is not 12-representable, so for example, the first two, and the last graphs in Figure~\ref{grid-graphs} are not 12-representable.

\begin{figure}[ht]
\begin{center}
\includegraphics[scale=0.9]{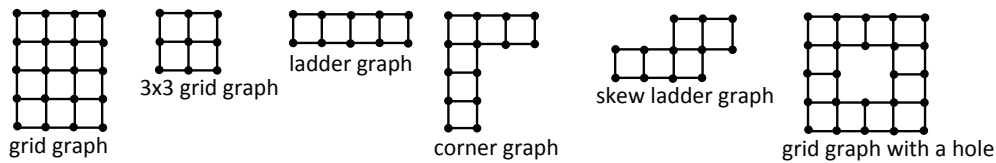}
\end{center}
\vspace{-10pt}
\caption{Induced subgraphs of a grid graph.}
\label{grid-graphs}
\end{figure}

Even though it was not stated explicitly, the concern of Jones at al.\ in \cite{JKPR15} was induced subgraphs of a grid graph that consist of a number of squares, which we call {\em square grid graphs}. In Section~\ref{sec:square} we provide a complete characterization in terms of forbidden induced subgraphs of 12-representable square grid graphs (see Theorem~\ref{thm:F=12re}). However, induced subgraphs of a grid graph may also contain edges (called by us ``lines'') that do not belong to any squares, for example, as in the graph in Figure~\ref{fig:induceG}. We refer to such subgraphs as {\em line grid graphs} (not to be confused with taking the line graph operation!) and think of the set of square grid graphs be disjoint with the set of line grid graphs. In Section~\ref{sec:generalgrid} we give a number of results on 12-representation of line grid graphs and state a conjecture on the complete characterization related to 12-representability in this case (see Conjecture~\ref{conj:linegrid2}).

\begin{figure}[!htbp]
\begin{center}
\begin{tikzpicture}[line width=0.7pt,scale=0.7]
\coordinate (O) at (0,0);
\draw [thick] (O)--++(0,1)--++(0,1)--++(1,0);
\draw [thick] (O)--++(1,0)++(1,0)--++(2,0)--++(0,-1)--++(-1,0)--++(0,1)
++(0,-1)--++(-1,0)--++(0,1)++(-1,0)--++(0,-1)--++(-1,0)--++(0,1);
\draw [thick] (O)++(3,0)--++(0,1)++(0,1)--++(0,1);
\draw [dotted, thick] (O)++(3,1)--++(0,1);
\draw [dotted, thick] (O)++(1,0)--++(1,0);
\draw [dotted, thick] (O)++(1,-1)--++(1,0);
\fill[black!100] (O) circle(0.5ex) ++(1,0) circle(0.5ex)
 ++(1,0) circle(0.5ex) ++(1,0) circle(0.5ex) ++(1,0) circle(0.5ex)
 ++(0,-1) circle(0.5ex) ++(-1,0) circle(0.5ex)++(-1,0) circle(0.5ex)
 ++(-1,0) circle(0.5ex)++(-1,0) circle(0.5ex);

 \fill[black!100] (O)++(3,1)circle(0.5ex) ++(0,1)circle(0.5ex)
 ++(0,1)circle(0.5ex);
 \fill[black!100] (O)++(0,1)circle(0.5ex) ++(0,1)circle(0.5ex)
 ++(1,0)circle(0.5ex);
\end{tikzpicture}
\caption{An example of a line grid graph}
\label{fig:induceG}
\end{center}
\end{figure}
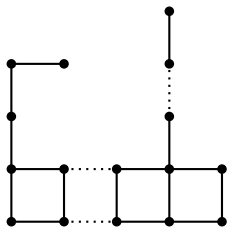

We conclude the introduction by reviewing some basic definitions and results given in \cite{JKPR15}, which will be used frequently in our paper.

A graph $H=(V',E')$ is an induced subgraph of $G=(V,E)$ if $V' \subset V$ and for all
$x,y \in V'$, $xy \in E'$ if and only if $xy \in E$. Also, similarly to the definition of $\red(w)$ for a word $w$, the reduced form $\red(H)$ of $H$ is obtained from the graph $H$ by replacing the $i$-th smallest label by $i$.

\begin{obs}[\cite{JKPR15}] \label{obs:induced}
If G is $12$-representable and $H=(V',E')$ is an induced subgraph of $G$, then $H$ is $12$-representable.
\end{obs}

\begin{thm}[\cite{JKPR15}]\label{at-most-twice-occur}
For a labeled $12$-representable graph $G$, there exists a word-representant $w$ in which each letter occurs at most twice. Also, $G$ can be represented by a permutation if and only if $G$ is a permutation graph.
\end{thm}

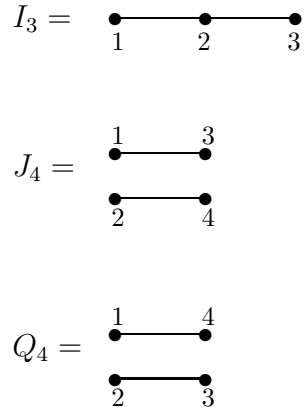
\begin{figure}[htp]
\begin{center}\setlength{\unitlength}{.06mm}

\begin{picture}(0,1000)(0,0)

\put(-70,950){\line(1,0){400}}
\put(-90,930){$\bullet$}
\put(110,930){$\bullet$}
\put(310,930){$\bullet$}
\put(-80,880){\footnotesize 1}
\put(110,880){\footnotesize 2}
\put(310,880){\footnotesize 3}
\put(-300,930){$I_3=$}

\put(-70,650){\line(1,0){200}}
\put(-90,630){$\bullet$}
\put(110,630){$\bullet$}
\put(-70,550){\line(1,0){200}}
\put(-90,530){$\bullet$}
\put(110,530){$\bullet$}
\put(-80,670){\footnotesize 1}
\put(120,670){\footnotesize 3}
\put(-80,490){\footnotesize 2}
\put(120,490){\footnotesize 4}
\put(-300,600){$J_4=$}

\put(-70,250){\line(1,0){200}}
\put(-90,230){$\bullet$}
\put(110,230){$\bullet$}
\put(-70,150){\line(1,0){200}}
\put(-90,130){$\bullet$}
\put(110,130){$\bullet$}
\put(-80,270){\footnotesize 1}
\put(120,270){\footnotesize 4}
\put(-80,90){\footnotesize 2}
\put(120,90){\footnotesize 3}
\put(-300,200){$Q_4=$}

\end{picture}
\caption{The graphs $I_3$, $J_4$ and $Q_4$}
\label{fig:IJQ}
\end{center}
\end{figure}

\begin{lem}[\cite{JKPR15}]\label{lem:IJQ}
Let $G=(V,E)$ be a labeled graph. If $G$ has an induced subgraph $H$
such that $\red(H)$ is equal to one of $I_3$, $J_4$ or $Q_4$  in Figure \ref{fig:IJQ}, then $G$
is not $12$-representable.
\end{lem}

\begin{deff}
A labeling of a graph is \emph{good} if it contains no induced subgraphs
equal to $I_3, J_4$ or $Q_4$ in the reduced form.
\end{deff}

For two sets of integers $A,B$, we write $A<B$ if every element of $A$ is less than each element in $B$.  Also, a subset $U$
of $V$ is called a {\em cutset} of $G=(V,E)$ if $G \setminus U$ is disconnected.

\begin{lem}[\cite{JKPR15}]\label{lem:cutset}
Let $G=(V,E)$ be a labeled graph and $U$ be a cutset of $G$.
Assume that $G_1=(V_1,E_1)$ and $G_2=(V_2,E_2)$ are two
components of $G \setminus U$. If $G$ is $12$-representable,
$|V_1| \geq 2$, $|V_2| \geq 2$, and the smallest element of $V_1 \cup V_2$ is in $V_1$, then $V_1 < V_2$.
\end{lem}

\begin{thm}[\cite{JKPR15}]\label{thm:cycle}
The cycle graph of length larger than $4$ is not $12$-representable.
\end{thm}

Finally, throughout this paper we assume that the graphs in question are connected since a graph is 12-representable if and only if each of its connected components is 12-representable.  Indeed, if $G$ is 12-representable then clearly each of its connected componets is 12-representable using the hereditary nature of 12-representation. Conversely, label the connected components $G_1$, $G_2,\ldots$ of a graph $G$ in a proper way, respectively, by $\{1,\ldots,|G_1|\}$, $\{|G_1|+1,\ldots,|G_1|+|G_2|\},\ldots$, and then use the respective word-representants $w_1$, $w_2,\ldots$ to obtain the word $w_1w_2\cdots$ 12-representing $G$.

\section{$12$-representability of square grid graphs}\label{sec:square}

Because the  grid graph given in Figure~\ref{fig:c4} can be
12-represented by $3412$, in all our arguments in this section, we
always assume that the  grid graphs are of size larger that $4$. Note that the labeling of the square graph in Figure~\ref{fig:c4} is the only good labeling up to rotation and swapping 1 and 2, and 3 and 4.

\begin{figure}[!htbp]
\begin{center}
\begin{tikzpicture}[line width=0.7pt,scale=0.7]
\coordinate (O) at (0,0);

\path ([c]O) node[left] {$3$}
++(0,1) node[left] {$1$}
++(1,0) node[right] {$4$}
++(0,-1) node[right] {$2$};
\fill[black!100] ([c]O) circle(0.5ex)
         ++(0,1)circle(0.5ex) ++(1,0)circle(0.5ex)
	      ++(0,-1)circle(0.5ex);
\draw[ thick] ([c]O)-- ++(0,1)-- ++(1,0)-- ++(0,-1)--++(-1,0);
\end{tikzpicture}
\caption{A labeled square grid graph with $4$ nodes}
\label{fig:c4}
\end{center}
\end{figure}
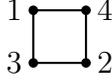

Let $F:=\{X\}\cup\{C_{2n}\}_{n \geq 4}$,
where $X$ is given in Figure~\ref{fig:corneradd1} and $C_{2n}$ is the cycle graph on $2n$ nodes. Then, we have the following lemma.

 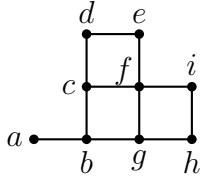
\begin{figure}[!htbp]
\begin{center}
\begin{tikzpicture}[line width=0.7pt,scale=0.7]
\coordinate (O) at (0,0);
\draw [thick] (O)--++(0,2)--++(1,0)--++(0,-1)--++(-1,0)++(0,-1)--++(2,0)
--++(0,1)--++(-1,0)--++(0,-1);
\fill[black!100] (O) circle(0.5ex) ++(0,1) circle(0.5ex)
 ++(0,1) circle(0.5ex) ++(1,0) circle(0.5ex) ++(0,-1) circle(0.5ex)
  ++(0,-1) circle(0.5ex) ++(1,0) circle(0.5ex)
  ++(0,1)circle(0.5ex);
\path (O)++(-1,0) node [left] {$a$};
\path (O) node[below] {$b$}
++(0,1) node[left] {$c$}
++(0,1) node[above] {$d$}
++(1,0) node[above] {$e$}
++(0,-2) node[below] {$g$}
++(1,0) node[below] {$h$}
++(0,1) node[above] {$i$};
\path (0.7,1.3) node {$f$};
\draw [thick] (O)--++(-1,0);
\fill[black!100] (O)++(-1,0) circle(0.5ex);
\end{tikzpicture}
\caption{The non-$12$-representable graph $X$}
\label{fig:corneradd1}
\end{center}
\end{figure}

\begin{lem}\label{lem:F-avoid}
If a graph $G=(V,E)$ has an induced subgraph $H$ in $F$, then $G$ is not $12$-representable.
\end{lem}

\begin{proof}
In view of Theorem~\ref{thm:cycle}, it suffices to show that $X$ is not 12-representable.

We prove this by showing that there is no good labeling for $X$.
If $1 \in \{a,b\}$, then viewing $\{c,f\}$ as a cutset,
we have $\{d,e\}>\{h,i\}$ by Lemma \ref{lem:cutset}. While, choosing $\{f,g\}$ as a cutset,
we have $\{h,i\}>\{d,e\}$, a contradiction. Hence, $1 \notin \{a,b\}$.
If $1 \in \{h,i\}$, then viewing $\{c,f\}$ as a cutset,
we have $\{d,e\}>\{a,b\}$. While, choosing $\{c,g\}$ as a cutset,
we have $\{a,b\}>\{d,e\}$, a contradiction. Hence, $1 \notin \{h,i\}$. By symmetry, $1 \notin \{d,e\}$.

Now, assume that $g=1$. To avoid $I_3$, $J_4$ and $Q_4$,
there are two choices for $2$, namely, $i=2$ or $a=2$.
If $i=2$, then it can be checked that $h=3$
and $c=4$. To avoid $I_3$, we have $5 \notin \{b,d,f\}$.
To avoid $J_4$ and $Q_4$, we have $5 \notin \{a,e\}$.
This means that there is no position for $5$. Hence,
$i \neq 2$.
Similarly, setting $a=2$ makes us find no place for $3$.
It follows that $g \neq 1$.
By symmetry, $c \neq 1$.

Assume that $f=1$, then $h=2$ or $d=2$. W.l.o.g. assume that
$h=2$, then $i=3$.
Then, to avoid $I_3$, $J_4$ and $Q_4$, there is no position for
$4$. Hence, we deduce that $f \neq 1$.

Thus, there exists no good labeling for $X$, which completes the proof.
\end{proof}

If a graph $G$ does not contain an induced subgraph in $F$, then we say that $G$ is {\em $F$-avoiding}. By Lemma~\ref{lem:F-avoid}, any 12-representable graph is $F$-avoiding.

\begin{deff}
A square $S$ in a grid graph is called an \emph{end-square} if it is
incident with an edge $ab$ such that neither $a$ nor $b$ is a corner
point of a square different from $S$.
\end{deff}

\begin{figure}[!htbp]
\begin{center}
\begin{tikzpicture}[line width=0.7pt,scale=0.6]
\def\dist{6}
\coordinate (BL) at (0,0) ; 
\coordinate (BLa) at (1,0.5);
\coordinate (BLL) at (0,-0.5);


\begin{scope}[every coordinate/.style={shift={(1*\dist,0)}}]
\draw[very thick] ([c]BL)--++(3,0)--++(0,1)--++(-3,0)--++(0,-1);
\draw[very thick] ([c]BL)++(1,0)--++(0,2)--++(1,0)--++(0,-2);
\fill[black!100] ([c]BL) circle(0.5ex) ++(1,0)circle(0.5ex) ++(1,0)circle(0.5ex)++(1,0)circle(0.5ex)++(0,1)circle(0.5ex)
++(-1,0)circle(0.5ex)
++(-1,0)circle(0.5ex)++(-1,0)circle(0.5ex)
++(1,1)circle(0.5ex) ++(1,0)circle(0.5ex);
\path ([c]BLL) node {$a$}
++(1,0)  node {$b$}
++(1,0) node {$c$}
++(1,0) node {$d$}
++(0,1.8) node[right] {$h$}
++(-1,0) node[right] {$g$}
++(-1,0) node[left] {$f$}
++(-1,0) node[left] {$e$}
++(1,1.1) node[left] {$i$}
++(1,0) node[right] {$j$};
\end{scope}

\begin{scope}[every coordinate/.style={shift={(2*\dist,0)}}]
\draw[very thick] ([c]BLa)++(1,0)--++(2,0)--++(0,1)--++(-2,0)--++(0,-1);
\draw[very thick] ([c]BLa)++(1,0)--++(0,2)--++(1,0)--++(0,-2);
\fill[black!100] ([c]BLa)  ++(1,0)circle(0.5ex) ++(1,0)circle(0.5ex)++(1,0)circle(0.5ex)++(0,1)circle(0.5ex)
++(-1,0)circle(0.5ex)
++(-1,0)circle(0.5ex)
++(0,1)circle(0.5ex) ++(1,0)circle(0.5ex);
\draw[very thick] ([c]BLa)++(1,0)--++(-1,0)--++(0,-1)--++(1,0)
--++(0,1);
\fill[black!100] ([c]BLa)  circle(0.5ex)
++(0,-1)circle(0.5ex)
++(1,0)circle(0.5ex);
\path ([c]BLa)++(0,-1.5) node {$a$}
++(1,0)  node {$d$}
++(0,1.8) node[left] {$c$}
++(-1,0) node[left] {$b$}
++(2,-0.7) node {$e$}
++(1,0) node {$f$}
++(0,1.7) node[right] {$k$}
++(-1,0.1) node[right] {$j$}
++(-1,-0.1) node[left] {$g$}
++(0,1) node[left] {$h$}
++(1,0) node[right] {$i$};
\end{scope}
\end{tikzpicture}
\caption{The graphs $G_1$ and $G_2$, respectively. Non-$12$-representability of these graphs follows from the fact that they contain $X$ as an induced subgraph.}
\label{fig:G123}
\end{center}
\end{figure}
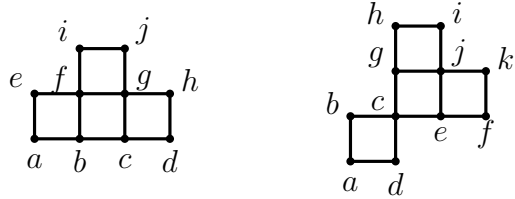

For example, in Figure~\ref{fig:G123}, all squares in $G_1$ but $bcgf$ are end-squares, and all squares in $G_2$ but $cgje$ are end-squares. In Lemma~\ref{lem:oneendsquare} below we will show that each $F$-avoiding square grid graph contains an end-square.

\begin{lem}\label{lem:posi-1}
Let $G=(V,E)$ be a labeled $12$-representable square grid graph, then
$1$ must be the label of a node of an end-square in $G$.
\end{lem}

\begin{proof}
 To prove this lemma, we assume to the contrary that $1$ is not a label of a node of an end-square of $G$. Since $G$ is $12$-representable, $G$ must be $F$-avoiding by Lemma \ref{lem:F-avoid}, so each node of $G$ belongs to at most three squares. We consider the following three cases.

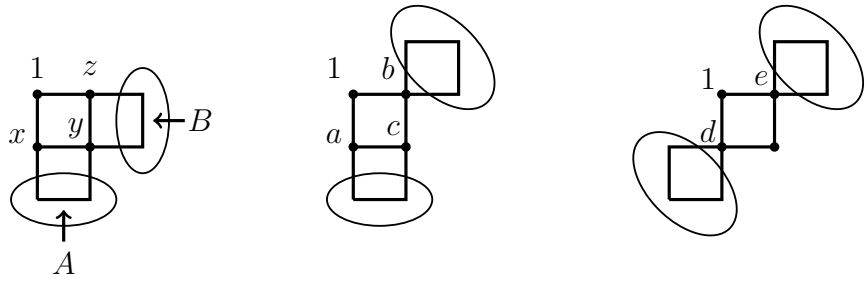
\begin{figure}[!htbp]
\begin{center}
\begin{tikzpicture}[line width=0.7pt,scale=0.7]
\def\dist{6}
\coordinate (BL) at (0,0) ; 
\coordinate (BL1) at (0.5,0);
\coordinate (BL2) at (0,1.2);
\coordinate (BL3) at (1,1);

\begin{scope}[every coordinate/.style={shift={(0,0)}}]
\draw[very thick] ([c]BL)-- ++(0,2)-- ++(2,0)-- ++(0,-1)-- ++(-2,0)
++(1,1)--++(0,-2)--++(-1,0);
\fill[black!100] ([c]BL)++(0,1) circle(0.5ex) ++(0,1)circle(0.5ex) ++(1,0)circle(0.5ex)++(0,-1)circle(0.5ex);
\draw ([c]BL1) ellipse (1 and 0.5);
\draw (2,1.5) ellipse (0.5 and 1);
\path ([c]BL2) node[left] {$x$}
++(0,1.3) node {$1$}
++(1,0) node {$z$}
++(0.1,-1.15) node[left] {$y$};
\draw [->,very thick] (0.5,-0.8) to (0.5,-0.2);
\path (0.5,-1.2) node {$A$};
\draw [->,very thick] (2.8,1.5) to (2.2,1.5);
\path (3.1,1.5) node {$B$};
\end{scope}

\begin{scope}[every coordinate/.style={shift={(1*\dist,0)}}]
\draw[very thick] ([c]BL)-- ++(0,2)-- ++(1,0)-- ++(0,-2)-- ++(-1,0)
++(0,1)--++(1,0)++(0,1)--++(0,1)--++(1,0)--++(0,-1)--++(-1,0);
\fill[black!100] ([c]BL)++(0,1) circle(0.5ex) ++(0,1)circle(0.5ex) ++(1,0)circle(0.5ex)++(0,-1)circle(0.5ex);
\path ([c]BL2) node[left] {$a$}
++(0,1.3) node[left] {$1$}
++(1,0) node[left] {$b$}
++(0.1,-1.15) node[left] {$c$};
\draw ([c]BL1) ellipse (1 and 0.5);
\draw  [rotate around={135:(7.7,2.7)}] (7.7,2.7) ellipse (1.2 and 0.7);
\end{scope}

\begin{scope}[every coordinate/.style={shift={(2*\dist,0)}}]
\draw[very thick] ([c]BL)-- ++(0,1)--++(1,0)--++(0,1)--++(1,0)
--++(0,1)--++(1,0)--++(0,-1)--++(-1,0)--++(0,-1)--++(-1,0)
--++(0,-1)--++(-1,0);
\draw  [rotate around={135:(14.7,2.7)}] (14.7,2.7) ellipse (1.2 and 0.7);
\draw  [rotate around={135:(12.3,0.3)}] (12.3,0.3) ellipse (1.2 and 0.7);
\path ([c]BL3) ++(0.1,0.3) node[left] {$d$}
++(0,1.0) node[left] {$1$}
++(1,0) node[left] {$e$};
\fill[black!100] ([c]BL3) circle(0.5ex) ++(0,1)circle(0.5ex) ++(1,0)circle(0.5ex)++(0,-1)circle(0.5ex);
\end{scope}
\end{tikzpicture}
\caption{Three subcases when $1$ belongs to only one square}
\label{fig:1one}
\end{center}
\end{figure}

\noindent
{\bf Case 1.} The node labeled by $1$ belongs to only one square. Possible situations in this case are given schematically in Figure~\ref{fig:1one}, where the ovals indicate the rest of the respective graphs. For the first subcase, it is easy to see that the number of nodes in $A$ or $B$ are  at least $2$. Choosing $\{x,y \}$ as a cutset, by Lemma~\ref{lem:cutset} we see that all labels of the nodes in $A$ are larger than those of $B$. While, choosing $\{y, z \}$ as a cutset, we have  all labels of the nodes in $B$ are larger than those of $A$, a contradiction. Hence, this subcase is impossible.
Similarly, by viewing $\{a,c\}$ and $\{b\}$ as a cutset respectively, we  prove that the second subcase is impossible. Viewing $\{d\}$ and $\{e\}$ as a cutset respectively, we
obtain that the third subcase is impossible as well.

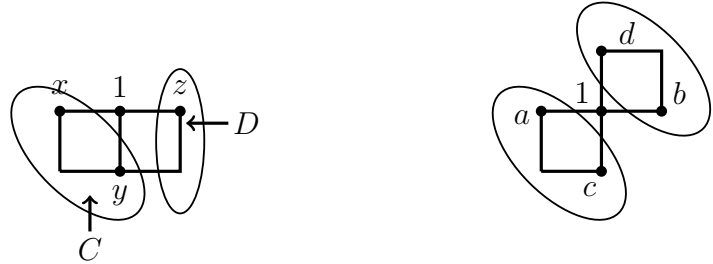
\begin{figure}[!htbp]
\begin{center}
\begin{tikzpicture}[line width=0.7pt,scale=0.8]
\def\dist{8}
\coordinate (BL) at (0,0) ; 

\begin{scope}[every coordinate/.style={shift={(0,0)}}]
\draw[very thick] ([c]BL)-- ++(0,1)-- ++(2,0)-- ++(0,-1)-- ++(-2,0)
++(1,0)--++(0,1);
\fill[black!100] ([c]BL)++(0,1) circle(0.5ex) ++(1,0)circle(0.5ex) ++(1,0)circle(0.5ex)++(-1,-1)circle(0.5ex);
\path ([c]BL)++(0,1.4) node {$x$}
++(1,0) node {$1$}
++(1,0) node {$z$}
++(-1,-1.8) node {$y$};
\draw  [rotate around={135:(0.3,0.3)}] (0.3,0.3) ellipse (1.4 and 0.7);
\draw   (2,0.5) ellipse (0.4 and 1.2);
\draw [->,very thick] (0.5,-1) to (0.5,-0.4);
\path (0.5,-1.3) node {$C$};
\draw [->,very thick] (2.8,0.8) to (2.1,0.8);
\path (3.1,0.8) node {$D$};
\end{scope}

\begin{scope}[every coordinate/.style={shift={(1*\dist,0)}}]
\draw[very thick] ([c]BL)-- ++(0,1)--++(2,0)--++(0,1)--++(-1,0)
--++(0,-2)--++(-1,0);
\draw  [rotate around={135:(9.7,1.7)}] (9.7,1.7) ellipse (1.4 and 0.7);
\draw  [rotate around={135:(8.3,0.3)}] (8.3,0.3) ellipse (1.4 and 0.7);
\path ([c]BL)++(0,0.9) node[left] {$a$}
++(1,0.4) node[left] {$1$}
++(1,0) node[right] {$b$}
++(-0.9,-1.6) node[left] {$c$}
++(0,2.6) node[right] {$d$};
\fill[black!100] ([c]BL)++(0,1) circle(0.5ex) ++(1,0)circle(0.5ex) ++(1,0)circle(0.5ex)++(-1,1)circle(0.5ex)++(0,-2)circle(0.5ex);
\end{scope}
\end{tikzpicture}
\caption{Two subcases when $1$ belongs to exactly two squares}
\label{fig:1two}
\end{center}
\end{figure}

\noindent
{\bf Case 2.} The node labeled by $1$ belongs to exactly two squares. Possible situations in this case are given in Figure~\ref{fig:1two}. For the first subcase, the number of nodes in $C$ is at least $5$ and
  the number of nodes in $D$ is at least $4$, which follows from
  the fact that the two squares $1$ belongs to are not end-squares.
  Choosing $\{x,y\}$ as a cutset, we see that all labels in $C$  except $x,y$ are  larger than those in $D$. Choosing $\{y,z\}$ as a cutset, we see that all labels in $C$ except $y$  are  smaller than those in $D$ except $z$, a contradiction. Thus, this subcase is impossible. Similarly, by viewing $\{a,c\}$ and $\{b,d\}$ as a cutset, respectively, we can prove that the second subcase is impossible.

\begin{figure}[!htbp]
\begin{center}
\begin{tikzpicture}[line width=0.7pt,scale=0.7]
\coordinate (BL) at (0,0) ; 

\draw[very thick] ([c]BL)-- ++(0,1)-- ++(1,0)-- ++(0,1)-- ++(1,0)
--++(0,-2)--++(-2,0)++(1,0)--++(0,1)--++(1,0);
\fill[black!100] ([c]BL)++(0,1) circle(0.5ex) ++(1,0)circle(0.5ex) ++(1,0)circle(0.5ex)++(-1,-1)circle(0.5ex)++(0,2)circle(0.5ex);
\path ([c]BL)++(0,1.3) node {$x$}
++(0.9,0) node[right] {$1$}
++(0.9,0) node[right] {$z$}
++(-1,-1.8) node {$y$}
++(-0.1,2.6) node {$k$};
\draw   (0,0.5) ellipse (0.5 and 1.1);
\draw   (1.5,2) ellipse (1.1 and 0.5);
\draw [->,very thick] (0,-0.9) to (0,-0.3);
\path (0,-1.2) node {$E$};
\draw [->,very thick] (2.9,2) to (2.3,2);
\path (3.2,2) node {$F$};
\end{tikzpicture}
\caption{The situation when $1$ belongs to three squares}
\label{fig:1three}
\end{center}
\end{figure}
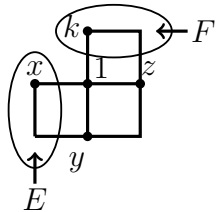

\noindent
{\bf Case 3.} The node labeled $1$ belongs to three squares. There is only one possibility here shown schematically in Figure~\ref{fig:1three}. Since the three squares
  $1$ belongs to are not end-squares, we see that the number of nodes in $E$, as well as that in $F$, is at least $4$. Choosing $\{x,y\}$ as a cutset, all labels in $E$ except $x$ are larger than those in $F$. On the other hand, choosing $\{z,k\}$ as a cutset, all labels in $E$  are larger than those in $F$ except $k$, a contradiction. Hence, this case is impossible.

Summarizing the cases above, we see that $1$ is never the label of a node of a non-end-square.
\end{proof}

\begin{lem}\label{lem:32}
Given a labeled $12$-representable square grid graph $G=(V,E)$ with an
end-square $S$ shown schematically in Figure \ref{fig:32}, with neither $2$ nor $3$ being a node of another square and $1$ being a node of another square, we can assume that a representant of $G$ is $w=3w_12w_212w_3$ with any letter in $w_1,w_2,w_3$ being at least $4$.
\end{lem}
\begin{figure}[!htbp]
\begin{center}
\begin{tikzpicture}[line width=0.7pt,scale=0.7]
\coordinate (O) at (0,0);

\path ([c]O) node[left] {$2$}
++(0,1) node[left] {$3$}
++(0.8,0.3) node[right] {$1$};
\fill[black!100] ([c]O) circle(0.5ex)
         ++(0,1)circle(0.5ex) ++(1,0)circle(0.5ex)
	      ++(0,-1)circle(0.5ex);
\draw[ thick] ([c]O)-- ++(0,1)-- ++(1,0)-- ++(0,-1)--++(-1,0)++(1,1)--++(0.4,0);
\draw   (1,0.5) ellipse (0.7 and 1.2);
\end{tikzpicture}
\caption{An end-square with $2$ and $3$ belonging to a single square}
\label{fig:32}
\end{center}
\end{figure}
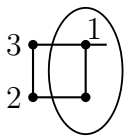

\begin{proof}
Let $w$ be a $12$-representant  of $G$.
Firstly, note that erasing all 3s in $w$ and placing a 3 at the beginning of the obtained word, we have a $12$-representant of $G$ since $3$ is only connected to $1$ and $2$ and no other connections are changed.

Secondly, we claim that we can assume that $1$ occurs only once in $w$. This can be verified by removing all but the leftmost $1$ in $w$. Since $1$ is the smallest letter, and the labels of the nodes connected to it must be to the left of
the leftmost $1$ in $w$, nothing will be changed after this operation.

Thirdly, we can assume that there are two copies of $2$ in $w$ and 1 is between the $2$s. Indeed, by Theorem~\ref{at-most-twice-occur}, we can assume that $2$ occurs at most twice in $w$, and clearly since there is no edge between $1$ and $2$, after $1$ there must be at least one $2$.  However, the other $2$ must be before $1$, since there exist nodes
connected to $1$ but not to $2$.

Lastly, we claim that the only $1$ and the $2$ after it can be assumed to be next to each other in $w$. If not, we can move all the letters between $1$ and the second $2$ right
after $2$ keeping their relative order. This operation is allowable since all letters between $1$ and $2$ must be not connected to $2$ because of the $2$ before $1$. Hence, moving the letters does not introduce any change. This completes the proof.
\end{proof}

\begin{prop}\label{lem:13}
Let $G=(V,E)$ be a $12$-representable square grid graph. Then, there exists a labeled copy $G'$ of $G$, which realizes $12$-representability of $G$, with
an end-square given in Figure \ref{fig:13}, with neither $1$ nor $3$ belonging to another square.
\end{prop}
\begin{figure}[!htbp]
\begin{center}
\begin{tikzpicture}[line width=0.7pt,scale=0.7]
\coordinate (O) at (0,0);
\path ([c]O) node[left] {$1$}
++(0,1) node[left] {$3$}
++(0.8,0.3) node[right] {$2$};
\fill[black!100] ([c]O) circle(0.5ex)
         ++(0,1)circle(0.5ex) ++(1,0)circle(0.5ex)
	      ++(0,-1)circle(0.5ex);
\draw[ thick] ([c]O)-- ++(0,1)-- ++(1,0)-- ++(0,-1)--++(-1,0);
\draw   (1,0.5) ellipse (0.5 and 1.2);
\end{tikzpicture}
\caption{An end-square with neither $1$ nor $3$ belonging to another square}
\label{fig:13}
\end{center}
\end{figure}
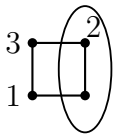
\begin{proof}
By Lemma \ref{lem:posi-1}, $1$ must be the label
of a node in an end-square. Then, for the labeling of the end-square, it is easy to check that there are
only two cases up to symmetry, which are given in Figures~\ref{fig:32}
and~\ref{fig:13}, respectively. If it is the case given in
 Figure~\ref{fig:13}, we are done. If not, by Lemma~\ref{lem:32} we assume that a  $12$-representant of $G$ is $w=3w_12w_212w_3$.
 Let $G''$ be the graph obtained from $G'$ by exchanging the labels $1$ and $2$. We claim that $G''$ is also $12$-representable and its representant is given by $\overline{w}=3w_11w_22w_3$. This can be verified by the fact that the nodes connected to both $1$ and $2$ in $G'$
 remain connected  to $1$ and $2$ in $G''$, while the nodes connected to
 $1$ but not $2$ in $G'$ become connected to  $2$ but not $1$ in $G''$. This completes the proof.
\end{proof}

As a by-product of the above lemma, we obtain the following corollary, which will be frequently used in the rest of the paper.

\begin{coro}\label{coro:represntant13}
Let $G=(V,E)$ be a $12$-representable square grid graph with an end-square given in Figure~\ref{fig:13}, then $G$ can be represented by some $w=3w_11w_22w_3$
with each letter in $w_1, w_2, w_3$ being at least $4$.
\end{coro}

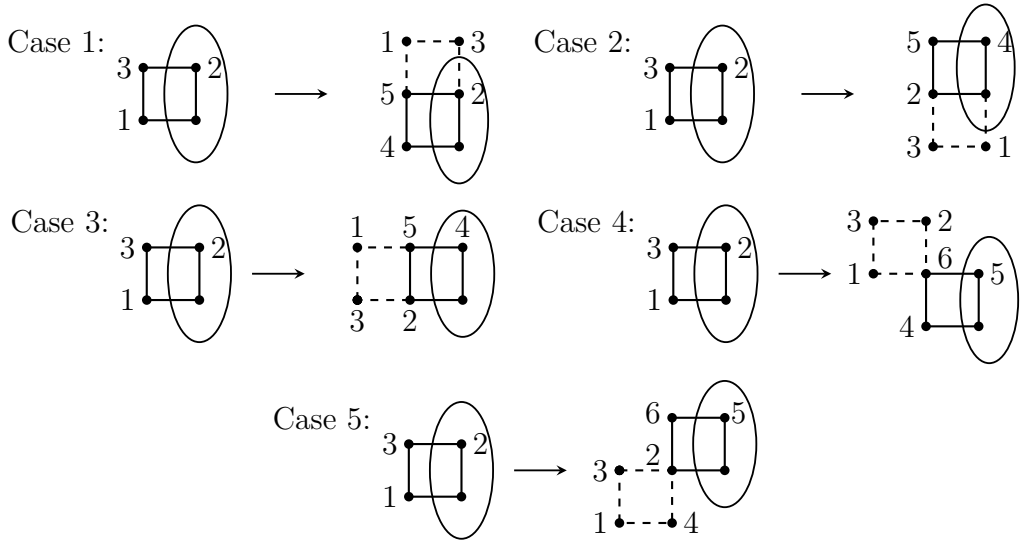
\begin{figure}[!htbp]
\begin{center}
\begin{tikzpicture}[line width=0.7pt,scale=0.7]
\tikzset{>=stealth};
\coordinate (bln) at (0,0);
\coordinate (B) at (2.5,0.5);
\coordinate (C) at (5,-0.5);

\coordinate (bln2) at (10,0);
\coordinate (B2) at (12.5,0.5);
\coordinate (C2) at (15,-0.5);

\path ([c]bln) node[left] {$1$}
++(0,1) node[left] {$3$}
++(1,0) node[right] {$2$}
++(-1.5,0.5) node[left] {Case 1:};
\fill[black!100] ([c]bln) circle(0.5ex)
         ++(0,1)circle(0.5ex) ++(1,0)circle(0.5ex)
	      ++(0,-1)circle(0.5ex);
\draw[ thick] ([c]bln)-- ++(0,1)-- ++(1,0)-- ++(0,-1)--++(-1,0);
\draw   (1,0.5) ellipse (0.6 and 1.3);

\draw[->] ([c]B)-- ++(1,0);

\path ([c]C) node[left] {$4$}
++(0,1) node[left] {$5$}
++(1,0) node[right] {$2$}
++(0,1) node[right] {$3$}
++(-1,0) node[left] {$1$};

\fill[black!100] ([c]C) circle(0.5ex)
         ++(0,1)circle(0.5ex) ++(1,0)circle(0.5ex)
	      ++(0,-1)circle(0.5ex) ++(0,2)circle(0.5ex)
          ++(-1,0)circle(0.5ex);
\draw[ thick] ([c]C)-- ++(0,1)-- ++(1,0)-- ++(0,-1)--++(-1,0);
\draw[dashed,thick] ([c]C)++(0,1)--++(0,1)--++(1,0)--++(0,-1);
\draw   (6,0) ellipse (0.55 and 1.2);
\path ([c]bln2) node[left] {$1$}
++(0,1) node[left] {$3$}
++(1,0) node[right] {$2$}
++(-1.5,0.5) node[left] {Case 2:};
\fill[black!100] ([c]bln2) circle(0.5ex)
         ++(0,1)circle(0.5ex) ++(1,0)circle(0.5ex)
	      ++(0,-1)circle(0.5ex);
\draw[ thick] ([c]bln2)-- ++(0,1)-- ++(1,0)-- ++(0,-1)--++(-1,0);
\draw   (11,0.5) ellipse (0.6 and 1.3);

\draw[->] ([c]B2)-- ++(1,0);

\path ([c]C2) node[left] {$3$}
++(0,1) node[left] {$2$}
++(1,0) 
++(0,1) node[right] {$4$}
++(-1,0) node[left] {$5$}
++(1,-2,0) node[right] {$1$};

\fill[black!100] ([c]C2) circle(0.5ex)
         ++(0,1)circle(0.5ex) ++(1,0)circle(0.5ex)
	      ++(0,-1)circle(0.5ex) ++(0,2)circle(0.5ex)
          ++(-1,0)circle(0.5ex);
\draw[dashed,thick] ([c]C2)-- ++(0,1)++(1,0)-- ++(0,-1)--++(-1,0);
\draw[thick] ([c]C2)++(0,1)--++(0,1)--++(1,0)--++(0,-1)--++(-1,0);
\draw   (16,1) ellipse (0.55 and 1.2);
\end{tikzpicture}
\vskip 5pt
\begin{tikzpicture}[line width=0.7pt,scale=0.7]
\tikzset{>=stealth};
\coordinate (bln) at (0,0);
\coordinate (B) at (2,0.5);
\coordinate (C) at (4,0);

\coordinate (bln2) at (10,0);
\coordinate (B2) at (12,0.5);
\coordinate (C2) at (13.8,0.5);

\path ([c]bln) node[left] {$1$}
++(0,1) node[left] {$3$}
++(1,0) node[right] {$2$}
++(-1.5,0.5) node[left] {Case 3:};
\fill[black!100] ([c]bln) circle(0.5ex)
         ++(0,1)circle(0.5ex) ++(1,0)circle(0.5ex)
	      ++(0,-1)circle(0.5ex);
\draw[ thick] ([c]bln)-- ++(0,1)-- ++(1,0)-- ++(0,-1)--++(-1,0);
\draw   (1,0.5) ellipse (0.6 and 1.3);

\draw[->] ([c]B)-- ++(1,0);

\draw[dashed, thick] ([c]C)--++(0,1)--++(1,0)++(0,-1)--++(-1,0);
\draw[thick] ([c]C)++(1,0)--++(0,1)--++(1,0)--++(0,-1)--++(-1,0);

\draw   (6,0.5) ellipse (0.6 and 1.2);

\path ([c]C)++(0,-0.4) node {$3$}
++(0,1.8) node {$1$}
++(1,0) node {$5$}
++(0,-1.8) node {$2$}
++(1,1.8) node  {$4$};

\fill[black!100] ([c]C) circle(0.5ex)
         ++(0,1)circle(0.5ex) ++(1,0)circle(0.5ex)
	      ++(1,0)circle(0.5ex) ++(0,-1)circle(0.5ex)
++(-1,0)circle(0.5ex) ++(-1,0)circle(0.5ex);

\path ([c]bln2) node[left] {$1$}
++(0,1) node[left] {$3$}
++(1,0) node[right] {$2$}
++(-1.5,0.5) node[left] {Case 4:};
\fill[black!100] ([c]bln2) circle(0.5ex)
         ++(0,1)circle(0.5ex) ++(1,0)circle(0.5ex)
	      ++(0,-1)circle(0.5ex);
\draw[ thick] ([c]bln2)-- ++(0,1)-- ++(1,0)-- ++(0,-1)--++(-1,0);
\draw   (11,0.5) ellipse (0.6 and 1.3);

\draw[->] ([c]B2)-- ++(1,0);

\draw[dashed,thick] ([c]C2)--++(1,0)--++(0,1)--++(-1,0)--++(0,-1);
\draw[thick] ([c]C2)++(1,0)--++(1,0)--++(0,-1)--++(-1,0)--++(0,1);
\path ([c]C2) node[left] {$1$}
++(0,1) node[left] {$3$}
++(1,0) node[right] {$2$}
++(0,-0.7) node[right] {$6$}
++(1,-0.3) node[right] {$5$}
++(-1,-1) node[left] {$4$};

\fill[black!100] ([c]C2) circle(0.5ex)
++(0,1)circle(0.5ex)++(1,0)circle(0.5ex)
++(0,-1)circle(0.5ex) ++(1,0)circle(0.5ex)
++(0,-1)circle(0.5ex)++(-1,0)circle(0.5ex);
\draw   (16,0) ellipse (0.55 and 1.2);
\end{tikzpicture}
\vskip 5pt
\begin{tikzpicture}[line width=0.7pt,scale=0.7]
\tikzset{>=stealth};
\coordinate (bln) at (0,0);
\coordinate (B) at (2,0.5);
\coordinate (C) at (4,0.5);

\path ([c]bln) node[left] {$1$}
++(0,1) node[left] {$3$}
++(1,0) node[right] {$2$}
++(-1.5,0.5) node[left] {Case 5:};
\fill[black!100] ([c]bln) circle(0.5ex)
         ++(0,1)circle(0.5ex) ++(1,0)circle(0.5ex)
	      ++(0,-1)circle(0.5ex);
\draw[ thick] ([c]bln)-- ++(0,1)-- ++(1,0)-- ++(0,-1)--++(-1,0);
\draw   (1,0.5) ellipse (0.6 and 1.3);

\draw[->] ([c]B)-- ++(1,0);

\draw[dashed, thick] ([c]C)--++(1,0)--++(0,-1)--++(-1,0)--++(0,1);
\draw[thick] ([c]C)++(1,0)--++(0,1)--++(1,0)--++(0,-1)--++(-1,0);

\draw   (6,1) ellipse (0.6 and 1.2);

\path ([c]C) node[left] {$3$}
++(0,-1) node[left] {$1$}
++(1,0) node[right] {$4$}
++(0,1.3) node[left] {$2$}
++(0,0.8) node[left] {$6$}
++(0.9,0) node[right] {$5$}
;

\fill[black!100] ([c]C) circle(0.5ex)
        ++(1,1)circle(0.5ex)
	      ++(1,0)circle(0.5ex) ++(0,-1)circle(0.5ex)
++(-1,0)circle(0.5ex) ++(-1,0)circle(0.5ex)
 ++(0,-1)circle(0.5ex)++(1,0)circle(0.5ex);
\end{tikzpicture}
\caption{Five ways in extending a $12$-representable square grid graph}
\label{fig:extend1}
\end{center}
\end{figure}


\begin{lem}\label{lem:extend}
Given a $12$-representable square grid graph $G$, we can obtain a new $12$-representable square grid graph $G'$ by extending $G$ in the
five ways presented schematically in Figure~\ref{fig:extend1}, where the extensions are applied to end-squares, which are indicated by dashed lines, and the assumption is that the dashed edges are the only new edges in the obtained graph.
\end{lem}

\begin{proof}
In view of Corollary \ref{coro:represntant13}, we may assume that
a $12$-representant of $G$ is $w=3 w_1 1 w_2 2 w_3$,
with each letter in $w_1, w_2, w_3$ being larger than $3$.
It is enough to find the corresponding
$12$-representant of $G'$ obtained under each extension.
Assume that $\pi'$ is obtain from $\pi$ by adding $2$ to each letter of $\pi$, while $\pi''$ is obtained from $\pi$ by adding $3$ to each letter of $\pi$. Then, it is not difficult to see that a representant of $G'$ for Cases 1--5 is, respectively,  $\overline{w}=3 5 1 w_1' 4 w_2' 2 4 w_3'$,  $\overline{w}=3 5  w_1' 1 5 2 w_2'  4 w_3'$, $\overline{w}=3 5  1 w_1' 2 w_2'  4 w_3'$, $\overline{w}=3 6  1 2 w_1'' 4 w_2''  5 w_3''$, and $\overline{w}=3 4 1 6 w_1'' 2 w_2''  5 w_3''$. Hence, in each of the five cases, $G'$ is 12-representable.
\end{proof}

\begin{lem}\label{lem:oneendsquare}
Each $F$-avoiding square grid graph $G$ has an end-square.
\end{lem}

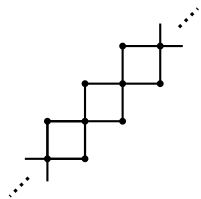
\begin{figure}[!htbp]
\begin{center}
\begin{tikzpicture}[line width=0.7pt,scale=0.5]
\coordinate (O) at (0,0);

\draw[ thick] ([c]O)-- ++(0,1)-- ++(1,0)
-- ++(0,1)-- ++(1,0) -- ++(0,1)-- ++(1,0)
-- ++(0,-1)-- ++(-1,0)-- ++(0,-1)-- ++(-1,0)
-- ++(0,-1)-- ++(-1,0)--++(-0.6,0) ++(0.6,0)
--++(0,-0.6)++(3,4.2)--++(0,-0.6)--++(0.6,0);

\fill[black!100] ([c]O) circle(0.5ex)
         ++(0,1)circle(0.5ex) ++(1,0)circle(0.5ex)
	      ++(0,-1)circle(0.5ex) ++(0,2)circle(0.5ex)++(1,0)circle(0.5ex)
	      ++(0,-1)circle(0.5ex)++(0,2)circle(0.5ex)++(1,0)circle(0.5ex)
	      ++(0,-1)circle(0.5ex);
\draw[ thick] ([c]bln)-- ++(0,1)-- ++(1,0)-- ++(0,-1)--++(-1,0);
\draw[ dotted, very thick] ([c]O)++(-0.5,-0.5)-- ++(-0.5,-0.5);
\draw[ dotted, very thick] ([c]O)++(3.5,3.5)-- ++(0.5,0.5);

\end{tikzpicture}
\caption{Grid graph with no common edges}
\label{fig:nocomedge}
\end{center}
\end{figure}

\begin{proof}
We consider two cases for a square grid graph $G$. If no two squares in $G$ have a common edge, then $G$ must be as in
Figure~\ref{fig:nocomedge}. Clearly, since $G$ is finite, it has an end-square. So, assume that at least two squares in $G$ have a common edge.
We start with one of such squares. If it is an end-square,
we are done. If it is not, we go to the
second square which has a common edge with it. If it is
an end-square, we are done. Otherwise, we go to the third one
which shares an edge with the second square but not the first one, if such a square exists.  If there is no such square, we choose the third square
to be the square that shares a node with the second square but not the first one.
Thus, squares with common edges always come first, and  then squares with common nodes. Continue in this fashion. Figure~\ref{fig:generate} illustrates possible cases  in the $(i+1)$-th step up to symmetry.

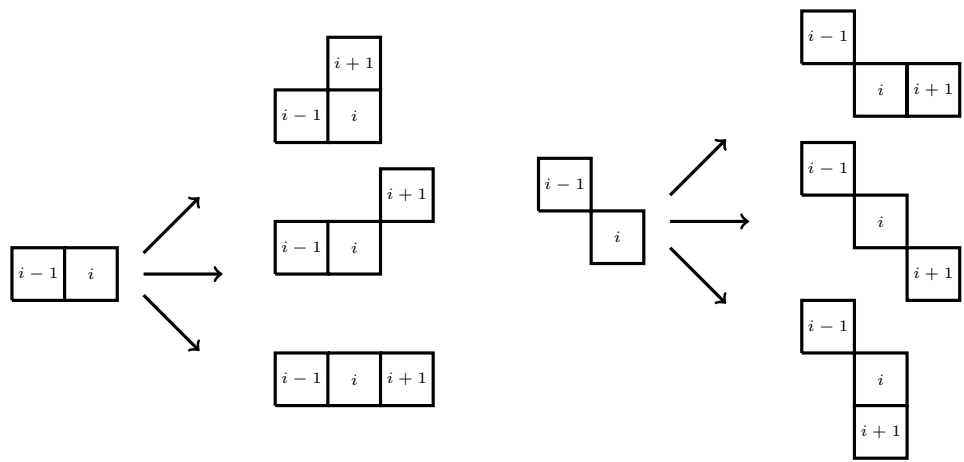
\begin{figure}[!htbp]
\begin{center}
\begin{tikzpicture}[line width=0.7pt,scale=0.7]
\coordinate (O) at (0,0);
\coordinate (B1) at (2.5,0.9);
\coordinate (B2) at (2.5,0.5);
\coordinate (B3) at (2.5,0.1);
\coordinate (C1) at (5,3);
\coordinate (C2) at (5,0.5);
\coordinate (C3) at (5,-2);
\coordinate (D) at (10,1.7);
\coordinate (E1) at (12.5,2);
\coordinate (E2) at (12.5,1.5);
\coordinate (E3) at (12.5,1);
\coordinate (F1) at (15,4.5);
\coordinate (F2) at (15,2);
\coordinate (F3) at (15,-1);

\draw[ very thick] ([c]O)-- ++(0,1)-- ++(1,0)--++(0,-1)--++(-1,0)
++(1,0)--++(1,0)--++(0,1)--++(-1,0);
\path ([c]O)++(0.5,0.5) node {{\tiny $i-1$}} ++(1,0) node {{\tiny $i$}};
\draw  [->, very thick, rotate around={45:(B1)}] (B1)--++(1.5,0);
\draw  [->, very thick] (B2)--++(1.5,0);
\draw  [->, very thick, rotate around={-45:(B1)}] (B3)--++(1.5,0);
\draw[ very thick] ([c]C1)-- ++(0,1)-- ++(1,0)--++(0,-1)--++(-1,0)
++(1,0)--++(1,0)--++(0,1)--++(-1,0)--++(0,1)--++(1,0)--++(0,-1);
\path ([c]C1)++(0.5,0.5) node {{\tiny $i-1$}} ++(1,0) node {{\tiny $i$}}
++(0,1) node {{\tiny $i+1$}};

\draw[ very thick] ([c]C2)-- ++(0,1)-- ++(1,0)--++(0,-1)--++(-1,0)
++(1,0)--++(1,0)--++(0,1)--++(-1,0)++(1,0)--++(1,0)--++(0,1)--++(-1,0)
--++(0,-1);
\path ([c]C2)++(0.5,0.5) node {{\tiny $i-1$}} ++(1,0) node {{\tiny $i$}}
++(1,1) node {{\tiny $i+1$}};

\draw[ very thick] ([c]C3)-- ++(0,1)-- ++(1,0)--++(0,-1)--++(-1,0)
++(1,0)--++(1,0)--++(0,1)--++(-1,0)++(1,0)--++(1,0)--++(0,-1)--++(-1,0);
\path ([c]C3)++(0.5,0.5) node {{\tiny $i-1$}} ++(1,0) node {{\tiny $i$}}
++(1,0) node {{\tiny $i+1$}};
\draw[ very thick] ([c]D)-- ++(0,1)-- ++(1,0)--++(0,-1)--++(-1,0)
++(1,0)--++(1,0)--++(0,-1)--++(-1,0)--++(0,1);
\path ([c]D)++(0.5,0.5) node {{\tiny $i-1$}} ++(1,-1) node {{\tiny $i$}};
\draw  [->, very thick, rotate around={45:(E1)}] (E1)--++(1.5,0);
\draw  [->, very thick] (E2)--++(1.5,0);
\draw  [->, very thick, rotate around={-45:(E3)}] (E3)--++(1.5,0);

\draw[ very thick] ([c]F1)-- ++(0,1)-- ++(1,0)--++(0,-1)--++(-1,0)
++(1,0)--++(1,0)--++(0,-1)--++(-1,0)--++(0,1)
++(1,0)--++(1,0)--++(0,-1)--++(-1,0)--++(0,1);
\path ([c]F1)++(0.5,0.5) node {{\tiny $i-1$}} ++(1,-1) node {{\tiny $i$}}
++(1,0) node {{\tiny $i+1$}};

\draw[ very thick] ([c]F2)-- ++(0,1)-- ++(1,0)--++(0,-1)--++(-1,0)
++(1,0)--++(1,0)--++(0,-1)--++(-1,0)--++(0,1)
++(1,-1)--++(1,0)--++(0,-1)--++(-1,0)--++(0,1);
\path ([c]F2)++(0.5,0.5) node {{\tiny $i-1$}} ++(1,-1) node {{\tiny $i$}}
++(1,-1) node {{\tiny $i+1$}};

\draw[ very thick] ([c]F3)-- ++(0,1)-- ++(1,0)--++(0,-1)--++(-1,0)
++(1,0)--++(1,0)--++(0,-1)--++(-1,0)--++(0,1)
++(1,-1)--++(0,-1)--++(-1,0)--++(0,1);
\path ([c]F3)++(0.5,0.5) node {{\tiny $i-1$}} ++(1,-1) node {{\tiny $i$}}
++(0,-1) node {{\tiny $i+1$}};
\end{tikzpicture}
\caption{Possible cases in the $(i+1)$-th step of searching for an end-square. Squares are labeled according to the step on which they were added.}
\label{fig:generate}
\end{center}
\end{figure}
Since we avoid the cycle graphs $C_{2n}$ for $n\geq 4$, the square labeled by $i+1$ will never be connected to
that labeled by a smaller number. Because the graph is finite,  sooner
or latter we will find an end-square.
\end{proof}

Actually, we can prove the following stronger version of Lemma~\ref{lem:oneendsquare}.

\begin{lem}\label{lem:twoendsquare}
In any $12$-representable square grid graph $G$ with more than
$4$ nodes, there exist exactly two end-squares.
\end{lem}

\begin{proof}
Firstly, we show that there exist at least two end-squares in $G$.
This can be done essentially by repeating the arguments in the proof of Lemma~\ref{lem:oneendsquare}. Begin with considering the square labeled by $1$. If it is an end-square, then there exists a square different from $1$ connected to it since $G$ has more than $4$ nodes.  Repeating the entire argument given in Lemma~\ref{lem:oneendsquare},
we can  find another end-square labeled by, say, $t$. If the square labeled by $1$ is not an end-square,
then there exists another square different from square $2$ connected to it, which we label by $t+1$. Repeating the entire argument given in Lemma~\ref{lem:oneendsquare} again, we obtain another end-square.

Next, we will show that the number of end-squares is no more than
two. Indeed, assume that $G$ has at least three end-squares, which implies that the number of nodes in $G$ is more than 12 as no two of end-squares can share a node to be $F$-avoiding; see Figure~\ref{fig:3endsquare} for a schematic representation of $G$. 

\begin{figure}[!htbp]
\begin{center}
\begin{tikzpicture}[line width=0.7pt,scale=0.7]
\coordinate (O) at (0,0);
\draw ([c]O) circle (1.5);
\draw [thick] (-0.5,1.4)--++(0,0.7)--++(1,0)--++(0,-0.7);
\draw [thick] (-1.07,-1.07)--++(-0.66,-0.45)--++(-0.53,0.87)--++(0.76,0.42);
\draw [thick] (1.07,-1.07)--++(0.66,-0.45)--++(0.53,0.87)--++(-0.76,0.42);
\fill[black!100] (-0.5,1.4) circle(0.5ex) ++(0,0.7)circle(0.5ex)
++(1,0)circle(0.5ex) ++(0,-0.7)circle(0.5ex);
\fill[black!100] (-1.07,-1.07) circle(0.5ex)++(-0.66,-0.45)circle(0.5ex)
++(-0.53,0.87)circle(0.5ex) ++(0.76,0.42)circle(0.5ex);
\fill[black!100] (1.07,-1.07) circle(0.5ex)++(0.66,-0.45)circle(0.5ex)
++(0.53,0.87)circle(0.5ex) ++(-0.76,0.42)circle(0.5ex);
\path (-0.45,1) node {$x_1$} ++(1,0) node {$x_2$}
++(0,1.5) node {$b$} ++(-1,0) node {$a$};
\path (0.7,-0.85) node {$x_4$} ++(0.4,0.75) node {$x_3$}
++(1.4,-0.7) node {$c$} ++(-0.55,-0.87) node {$d$};
\path (-0.7,-0.85) node {$x_5$} ++(-0.4,0.75) node {$x_6$}
++(-1.5,-0.7) node {$f$} ++(0.55,-0.87) node {$e$};
\end{tikzpicture}
\caption{Square grid graph with at least three end-squares}
\label{fig:3endsquare}
\end{center}
\end{figure}
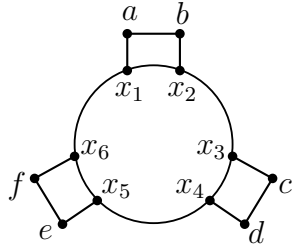

%

By Lemma \ref{lem:posi-1},  $1$ must be the label of a node of an end-square. W.l.o.g. assume that $1\in\{a,b,x_1,x_2\}$. Viewing $\{x_3,x_4\}$ as a cutset, we deduce that
$\{c,d\}>\{e,f\}$ by Lemma \ref{lem:cutset}. On the other hand, if we choose  $\{x_5,x_6\}$ as a cutset, we obtain that $\{c,d\}<\{e,f\}$, a contradiction. Thus, there is no place for the label $1$ in $G$ assuming $G$ has at least three end-squares, so $G$ must have at most two squares.
\end{proof}

Let $u=u_1 \cdots u_j \in \{1, \ldots, n\}^*$ and $\red(u)=u$.
Let the \emph{reverse} of $u$ be the word $u^r=u_j u_{j-1}\cdots u_1$
and the \emph{complement} of $u$ be the word $u^c=(n+1-u_1) \cdots (n+1-u_j)$.
Jones et al.~\cite{JKPR15} introduced the definition
of the \emph{supplement} of a graph. Given a graph $G=(\{1, \ldots, n\},E)$, let the supplement of $G$ be defined by $\overline{G}=(V,\overline{E})$
where for all $x,y \in V$, $xy \in E$ if and only if $n+1-x$ is adjacent to $n+1-y$ in $\overline{E}$. One can think of the supplement of $G=(V,E)$
as a relabeling of $G$ by replacing each label $x$ by the label $n+1-x$.

The following observation can be obtained by combining Observations 2 and 3 in \cite{JKPR15}.

\begin{obs}\label{obs:supplement}
Let $G=(V,E)$ be a $12$-representable graph and $w$ be a $12$-representant of $G$.
Then, $\overline{G}$ is also $12$-representable with a $12$-representant $(w^{r})^c$.
\end{obs}

\begin{coro}\label{coro:13nn-2}
Let $G=(V,E)$ be a $12$-representable square grid graph with $|V|=n$. Then, there exists a labeled graph $G'$, which realizes $12$-representability of $G$, with
one end-square given in Figure~\ref{fig:13} and the other end-square
given in Figure~\ref{fig:nn-2}.
\end{coro}

\begin{proof} By Proposition~\ref{lem:13}, there exists a labeled copy $G'$ of $G$, which realizes $12$-representability of $G$ with one end-square given in Figure \ref{fig:13}. By Observation~\ref{obs:supplement}, we can relabel $G$ so that  it remains $12$-representable with an end square being as in in Figure~\ref{fig:nn-2}. By Lemma~\ref{lem:twoendsquare}, there is another end-square in $G$, and by Lemma~\ref{lem:posi-1} 1 must the label of one of its nodes. If this endpoint as in Figure~\ref{fig:13}, we are done. Otherwise, it must be like in Figure~\ref{fig:32}, and the arguments in the proof of Proposition~\ref{lem:13} can be applied to relabel the square of the form in Figure~\ref{fig:32} into that in Figure~\ref{fig:13}. We are done. \end{proof}

\begin{figure}[!htbp]
\begin{center}
\begin{tikzpicture}[line width=0.7pt,scale=0.7]
\coordinate (O) at (0,0);
\draw[ thick] ([c]O)-- ++(0,1)-- ++(1,0)-- ++(0,-1)--++(-1,0);
\draw   (O)++(0,0.5) ellipse (0.6 and 1.5);
\path (O)++(0,1.4) node {{\tiny $n-1$}}
++(1.2,0) node  {{\tiny $n-2$}}
++(-0.1,-1.8) node  {{\tiny $n$}};
\fill[black!100] ([c]O) circle(0.5ex)
         ++(0,1)circle(0.5ex) ++(1,0)circle(0.5ex)
	      ++(0,-1)circle(0.5ex);
\end{tikzpicture}
\caption{An end-square with neither $n$ nor $n-2$ being a node of another square}
\label{fig:nn-2}
\end{center}
\end{figure}
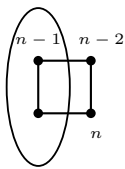

Our main result in this section is the following characterization theorem.

\begin{thm}\label{thm:F=12re}
A square grid graph is $F$-avoiding if and only if it is $12$-representable.
\end{thm}

\begin{proof}
Following from  Oberservation \ref{obs:induced} and Lemma \ref{lem:F-avoid}, it is easy to check that $12$-representable square grid graph is $F$-avoiding.
In the following, we proceed to  show that  an $F$-avoiding square grid graph is $12$-representable by induction on the number of squares.

Firstly, if a square grid graph $G$ has at most two squares, which is certainly $F$-avoiding, then it can be easily checked that $G$ is $12$-representable as $G$ is then either a square, or a two-square ladder graph, or two squares joint in a node. So, we assume that any $F$-avoiding square grid graph with $n$, $n>2$, squares
is $12$-representable. We wish to show that this still holds for an $F$-avoiding square grid graph $G$ with $n+1$ squares.

Since $G$ is $F$-avoiding, it must contain an end-square $S_1$ by  Lemma~\ref{lem:oneendsquare}. Let $G'$ be a graph obtained
form $G$ by removing $S_1$. Clearly, $G'$ is $F$-avoiding, which implies that $G'$ is $12$-representable by the induction hypothesis.
We claim that removing $S_1$ from $G$ produces a new end-suquare $S_2$ which is connected to $S_1$.  Since $n>2$,  the number of nodes of $G$ is not less than $8$.  We see that $S_2$ cannot be an end-square of $G$
since $G$ is connected. Then, it is easy to check this
claim through the following five cases up to symmetry. We omit the details here.

\begin{figure}[!htbp]
\begin{center}
\begin{tikzpicture}[line width=0.7pt,scale=0.7]
\coordinate (O) at (0,0);
\coordinate (B) at (3.5,0);
\coordinate (C) at (8,0);
\coordinate (D) at (11,0);
\coordinate (E) at (14.5,0);

\draw [thick] (O)--++(2,0)--++(0,2)--++(-1,0)--++(0,-1)--++(1,0)
++(-1,0)--++(-1,0)--++(0,-1)++(1,0)--++(0,1);
\fill[black!100] (O) circle(0.5ex) ++(1,0) circle(0.5ex)
++(1,0) circle(0.5ex) ++(0,1) circle(0.5ex) ++(-1,0) circle(0.5ex)
++(-1,0) circle(0.5ex)++(1,1)circle(0.5ex)++(1,0) circle(0.5ex);
\path (1.5,0.5) node {$S_2$}
++(0,1) node {$S_1$};
\draw [thick] (B)--++(2,0)--++(0,1)--++(-2,0)--++(0,-1)
++(1,0)--++(0,1)++(1,0)--++(1,0)--++(0,1)--++(-1,0)--++(0,-1);
\fill[black!100] (B) circle(0.5ex) ++(1,0) circle(0.5ex)
++(1,0) circle(0.5ex) ++(0,1) circle(0.5ex)
++(-1,0) circle(0.5ex)++(-1,0) circle(0.5ex)
++(3,0)circle(0.5ex) ++(0,1)circle(0.5ex) ++(-1,0)circle(0.5ex);
\path (B)++(1.5,0.5) node {$S_2$}
++(1,1) node {$S_1$};
\draw [thick] (C)--++(1,0)--++(0,3)--++(-1,0)--++(0,-3)
++(0,1)--++(1,0)++(0,1)--++(-1,0);
\fill[black!100] (C) circle(0.5ex)++(1,0)circle(0.5ex)
++(0,1)circle(0.5ex)++(0,1)circle(0.5ex)++(0,1)circle(0.5ex)
++(-1,0)circle(0.5ex) ++(0,-1)circle(0.5ex)++(0,-1)circle(0.5ex)++(0,-1)circle(0.5ex);
\path (C)++(0.5,1.5) node {$S_2$}
++(0,1) node {$S_1$};
\draw [thick] (D)--++(1,0)--++(0,1)--++(-1,0)--++(0,-1)++(1,1)
--++(1,0)--++(0,2)--++(-1,0)
--++(0,-2)++(0,1)--++(1,0);
\fill[black!100] (D) circle(0.5ex)++(1,0)circle(0.5ex)
++(0,1)circle(0.5ex)++(0,1)circle(0.5ex)++(0,1)circle(0.5ex)
++(1,0)circle(0.5ex) ++(0,-1)circle(0.5ex)++(0,-1)circle(0.5ex)++(-2,0)circle(0.5ex);
\path (D)++(1.5,1.5) node {$S_2$}
++(0,1) node {$S_1$};
\draw [thick] (E)--++(1,0)--++(0,1)--++(1,0)--++(0,1)
--++(1,0)--++(0,1)--++(-1,0)--++(0,-1)--++(-1,0)--++(0,-1)
--++(-1,0)--++(0,-1);
\fill[black!100] (E) circle(0.5ex)++(1,0)circle(0.5ex)
++(-1,1)circle(0.5ex)++(1,0)circle(0.5ex)++(1,0)circle(0.5ex)
++(-1,1)circle(0.5ex)++(1,0)circle(0.5ex)++(1,0)circle(0.5ex)
++(-1,1)circle(0.5ex)++(1,0)circle(0.5ex);
\path (E)++(1.5,1.5) node {$S_2$}
++(1,1) node {$S_1$};
\end{tikzpicture}
\caption{Five cases up to symmetry when removing an end-square of $G$ }
\label{fig:5cases}
\end{center}
\end{figure}
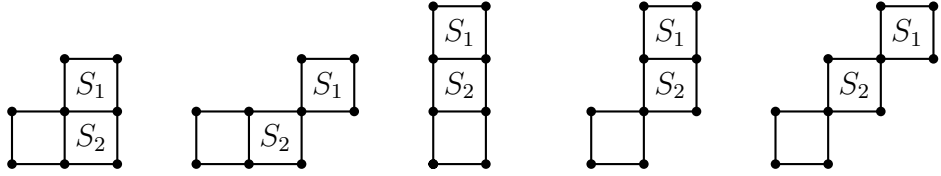
By Corollary \ref{coro:13nn-2}, there exists a labeling with
square $S_2$ being labeled as given in Figure~\ref{fig:13}
or Figure~\ref{fig:nn-2}. Viewing Obeservation \ref{obs:supplement},
we can always assume that $S_2$ is labeled as given in Figure \ref{fig:13}, otherwise, we will take its supplement to obtain a
labeling of $S_2$ as in Figure \ref{fig:13}. Thus, by Lemma \ref{lem:extend}, we can simply find a labeling of $G$, as well as its $12$-representant. This shows that
$G$ is $12$-representable and the proof is completed.
\end{proof}

\begin{remark} We note that Theorem~\ref{thm:F=12re} is still true if in $F$ the graph $X$ will be replaced by the square grid graphs $G_1$ and $G_2$ in Figure~\ref{fig:G123}. However, such a change would not result in a set of {\em minimal} forbidden induced subgraphs.
\end{remark}

By Theorem~\ref{thm:F=12re},  if a square grid graph is
$F$-avoidng, then it is 12-representable, and thus it has a good labeling. On the other hand, by Lemma~\ref{lem:F-avoid} we know that
if a square grid graph has a good labeling, then it is $F$-avoiding, and thus it is 12-representable by Theorem~\ref{thm:F=12re}. This leads us to the following corollary, which is still an open question in the case of arbitrary (not necessarily grid) graphs.

\begin{coro}\label{coro:googlabeling-12}
A square grid graph has a good labeling if and only if it is $12$-representable.
\end{coro}

In fact, we can prove a stronger version of Corollary~\ref{coro:googlabeling-12}, which is again an open question in the case of arbitrary graphs (no counter-example to this statement is known).

\begin{thm}\label{thm:anygoodlabeling-12}
For any good labeling of a square grid graph $G$, there exists a
word $w$ $12$-representing $G$.
\end{thm}

\begin{proof}
We prove, by induction on the number of squares, even a stronger statement
that any such $w$ begins with the only occurrence of the third smallest letter.

The base case is given in  Figure \ref{fig:c4}, and its $12$-representant is $3412$. Now, assume that this theorem holds for
square grid graphs with $n$ squares. We wish to show that it still holds
for $n+1$ squares. Notice that in a good labeling of $G$, $1$ is always  the label
of a node in an end-square. Hence, we need to consider three cases in Figure~\ref{fig:3casesgoodlabeling}, where the labels arrowed to each other can be swapped.

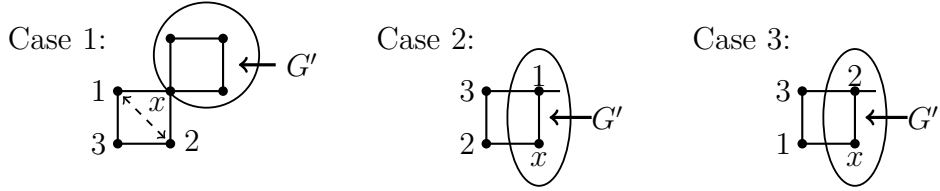
\begin{figure}[!htbp]
\begin{center}
\begin{tikzpicture}[line width=0.7pt,scale=0.7]
\coordinate (O) at (0,0);
\coordinate (B) at (7,0);
\coordinate (C) at (13,0);
\draw[ thick] ([c]O)-- ++(0,1)-- ++(1,0)-- ++(0,-1)--++(-1,0)
++(1,1) -- ++(0,1)-- ++(1,0)-- ++(0,-1)--++(-1,0);
\draw   (O)++(1.685,1.685) circle (1);
\fill[black!100] ([c]O) circle(0.5ex)
     ++(0,1)circle(0.5ex) ++(1,0)circle(0.5ex)
     ++(0,-1)circle(0.5ex)
     ++(0,2)circle(0.5ex) ++(1,0)circle(0.5ex)
     ++(0,-1)circle(0.5ex);
\path (O) node[left] {$3$}
++(0,1) node[left]  {$1$}
++(0.75,-0.25) node {$x$}
++(0.3,-0.7) node[right] {$2$};
\draw [->,very thick] (3,1.5) to (2.3,1.5);
\path (3.5,1.5) node {$G'$};
\draw [<->, dashed]  (0.1,0.9) to (0.9,0.1);
\path ++(0,2) node[left] {Case 1:};
\draw[thick] (B)--++(0,1)--++(1,0)--++(0,-1)--++(-1,0);
\draw[thick] (B)++(1,1)--++(0.4,0);
\fill[black!100] ([c]B) circle(0.5ex) ++(0,1)circle(0.5ex) ++(1,0)circle(0.5ex) ++(0,-1)circle(0.5ex);
\path (B) node[left] {$2$}
++(0,1) node[left] {$3$}
++(1,0.3) node {$1$}
++(0,-1.6) node{$x$};
\draw   (B)++(1,0.5) ellipse (0.6 and 1.3);
\draw [->,very thick]  (9,0.5) to (8.2,0.5);
\path (9.3,0.5) node {$G'$};
\path ++(7,2) node[left] {Case 2:};
\draw[thick] (C)--++(0,1)--++(1,0)--++(0,-1)--++(-1,0);
\draw[thick] (C)++(1,1)--++(0.4,0);
\fill[black!100] ([c]C) circle(0.5ex) ++(0,1)circle(0.5ex) ++(1,0)circle(0.5ex) ++(0,-1)circle(0.5ex);
\path (C) node[left] {$1$}
++(0,1) node[left] {$3$}
++(1,0.3) node {$2$}
++(0,-1.6) node{$x$};
\draw   (C)++(1,0.5) ellipse (0.6 and 1.3);
\draw [->,very thick]  (15,0.5) to (14.2,0.5);
\path (15.3,0.5) node {$G'$};
\path ++(13,2) node[left] {Case 3:};
\end{tikzpicture}
\caption{Three cases  for a gooding labeling square grid graph}
\label{fig:3casesgoodlabeling}
\end{center}
\end{figure}

For Case 1,  let $G$  be the graph  in Figure~\ref{fig:ex67} with exactly two squares. Then,  its $12$-representant is $36712645$, which is what we need to show.

\begin{figure}[!htbp]
\begin{center}
\begin{tikzpicture}[line width=0.7pt,scale=0.7]
\coordinate (O) at (0,0);
\draw[thick] (O)--++(0,1)--++(1,0)--++(0,-1)--++(-1,0)
++(1,1)--++(0,1)--++(1,0)--++(0,-1)--++(-1,0);
\fill[black!100] ([c]O) circle(0.5ex) ++(0,1)circle(0.5ex)
++(1,0)circle(0.5ex) ++(0,-1)circle(0.5ex)++(0,2)circle(0.5ex)
++(1,0)circle(0.5ex) ++(0,-1)circle(0.5ex);
\path (O) node[left] {$3$}
++(0,1) node[left] {$1$}
++(0.75,0.35) node {$7$}
++(0.25,-1.35) node[right] {$2$}
++(0,2.3) node[left] {$4$}
++(1,0) node[right] {$6$}
++(0,-1.3) node[right] {$5$};
\draw [<->, dashed]  (0.1,0.9) to (0.9,0.1);
\draw [<->, dashed]  (1.1,1.9) to (1.9,1.1);
\end{tikzpicture}
\caption{A special subcase in Case 1}
\label{fig:ex67}
\end{center}
\end{figure}
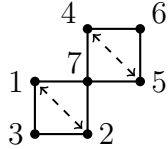

 Otherwise, we claim that $x$ is always the third smallest label in $G'$. Indeed, consider four possible subcases in Figure~\ref{fig:case1-4sub}. Remove the nodes labeled $1,2$ and $3$. Then, by induction hypothesis, $G'$ can be
represented by $w'=6w''$. But then, it can be checked that $G$ is represented by $w=3612w''$.

 \begin{figure}[!htbp]
\begin{center}
\begin{tikzpicture}[line width=0.7pt,scale=0.7]
\coordinate (O) at (0,0);
\coordinate (B) at (4,0);
\coordinate (C) at (9,0);
\coordinate (D) at (13,0);
\draw[ thick] ([c]O)-- ++(0,1)-- ++(1,0)-- ++(0,-1)--++(-1,0)
++(1,1) -- ++(0,1)-- ++(1,0)-- ++(0,-1)--++(-1,0);
\fill[black!100] ([c]O) circle(0.5ex)
     ++(0,1)circle(0.5ex) ++(1,0)circle(0.5ex)
     ++(0,-1)circle(0.5ex)
     ++(0,2)circle(0.5ex) ++(1,0)circle(0.5ex)
     ++(0,-1)circle(0.5ex);
\path (O) node[left] {$3$}
++(0,1) node[left]  {$1$}
++(0.7,-0.3) node {$6$}
++(0.3,-0.7) node[right] {$2$}
++(0,2) node[left] {$4$}
++(1,0) node[right] {$7$}
++(0,-1) node[right] {$5$};
\draw [<->, dashed]  (0.1,0.9) to (0.9,0.1);
\draw [<->, dashed]  (1.1,1.9) to (1.9,1.1);
\draw[thick] (B)--++(0,1)--++(1,0)--++(0,-1)--++(-1,0)
++(1,1)--++(2,0)--++(0,1)--++(-2,0)--++(0,-1)++(1,0)--++(0,1);
\fill[black!100] ([c]B) circle(0.5ex) ++(0,1)circle(0.5ex) ++(1,0)circle(0.5ex) ++(0,-1)circle(0.5ex)++(0,2)circle(0.5ex)
++(1,0)circle(0.5ex)++(1,0)circle(0.5ex)++(0,-1)circle(0.5ex)
++(-1,0)circle(0.5ex);
\path (B) node[left] {$3$}
++(0,1) node[left] {$1$}
++(0.7,-0.3) node {$6$}
++(0.3,-0.7) node[right] {$2$}
++(0,2) node[left] {$4$}
++(1,-1.4) node {$5$};
\draw [<->, dashed]  (5.1,1.9) to (5.9,1.1);
\draw [<->, dashed]  (4.1,0.9) to (4.9,0.1);
\draw   [rotate around={45:(6.7,1.7)}]  (6.7,1.7) ellipse (0.6 and 1.3);
\draw[thick] (C)--++(0,1)--++(1,0)--++(0,-1)--++(-1,0)
++(1,1)--++(0,2)--++(1,0)--++(0,-2)--++(-1,0)++(0,1)--++(1,0);
\fill[black!100] ([c]C) circle(0.5ex) ++(0,1)circle(0.5ex) ++(1,0)circle(0.5ex) ++(0,-1)circle(0.5ex)++(0,2)circle(0.5ex)
++(0,1)circle(0.5ex)++(1,0)circle(0.5ex)++(0,-1)circle(0.5ex)
++(0,-1)circle(0.5ex);
\path (C) node[left] {$3$}
++(0,1) node[left] {$1$}
++(0.7,-0.3) node {$6$}
++(0.3,-0.7) node[right] {$2$}
++(0,2) node[left] {$4$}
++(1,-1.4) node {$5$};
\draw [<->, dashed]  (10.1,1.9) to (10.9,1.1);
\draw [<->, dashed]  (9.1,0.9) to (9.9,0.1);
\draw   [rotate around={45:(10.7,2.7)}] (10.7,2.7) ellipse (0.6 and 1.3);
\draw[thick] (D)--++(0,1)--++(1,0)--++(0,-1)--++(-1,0)
++(1,1)--++(0,1)--++(1,0)--++(0,-1)--++(-1,0)
++(1,1)--++(0,1)--++(1,0)--++(0,-1)--++(-1,0);
\fill[black!100] ([c]D) circle(0.5ex) ++(0,1)circle(0.5ex) ++(1,0)circle(0.5ex) ++(0,-1)circle(0.5ex)++(0,2)circle(0.5ex)
++(1,0)circle(0.5ex) ++(0,-1)circle(0.5ex)
++(0,2)circle(0.5ex)
++(1,0)circle(0.5ex) ++(0,-1)circle(0.5ex);
\path (D) node[left] {$3$}
++(0,1) node[left] {$1$}
++(0.7,-0.3) node {$6$}
++(0.3,-0.7) node[right] {$2$}
++(0,2) node[left] {$4$}
++(1,-1.4) node {$5$};
\draw [<->, dashed]  (14.1,1.9) to (14.9,1.1);
\draw [<->, dashed]  (13.1,0.9) to (13.9,0.1);
\draw   [rotate around={45:(15.7,2.7)}] (15.7,2.7) ellipse (0.6 and 1.3);
\end{tikzpicture}
\caption{Four other  subcases in Case 1}
\label{fig:case1-4sub}
\end{center}
\end{figure}
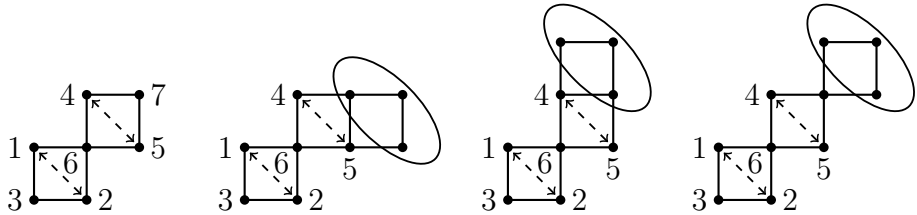

For Case 2, if $x$ is the third smallest label in $G'$, then
$w'=xw''$ represents $G'$ and $3x12w''$ represents $G$.
Now, we suppose  that $x$ is not the third smallest in $G'$. Consider
four possible subcases given in Figure~\ref{fig:case2-4sub}.

\begin{figure}[!htbp]
\begin{center}
\begin{tikzpicture}[line width=0.7pt,scale=0.7]
\coordinate (O) at (0,0);
\coordinate (B) at (4.5,0);
\coordinate (C) at (9,0);
\coordinate (D) at (14,0);
\draw[ thick] ([c]O)-- ++(0,1)-- ++(1,0)-- ++(0,-1)--++(-1,0)
++(1,1) -- ++(0,1)-- ++(1,0)-- ++(0,-1)--++(-1,0);
\fill[black!100] ([c]O) circle(0.5ex)
     ++(0,1)circle(0.5ex) ++(1,0)circle(0.5ex)
     ++(0,-1)circle(0.5ex)
     ++(0,2)circle(0.5ex) ++(1,0)circle(0.5ex)
     ++(0,-1)circle(0.5ex);
\path (O) node[left] {$2$}
++(0,1) node[left]  {$3$}
++(0.75,-0.25) node {$1$}
++(0.3,-0.7) node[right] {$4$};
\path ++(0,2) node[left] {$(1)$};
\draw [<->, dashed]  (0.1,0.9) to (0.9,0.1);
\draw  [rotate around={45:(1.7,1.7)}] (1.7,1.7) ellipse (0.6 and 1.3);
\draw[ thick] ([c]B)-- ++(0,1)-- ++(1,0)-- ++(0,-1)--++(-1,0)
++(1,0) -- ++(0,1)-- ++(1,0)-- ++(0,-1)--++(-1,0);
\fill[black!100] ([c]B) circle(0.5ex)
     ++(0,1)circle(0.5ex) ++(1,0)circle(0.5ex)
     ++(0,-1)circle(0.5ex)
     ++(1,0)circle(0.5ex) ++(0,1)circle(0.5ex);
\path (B) node[below] {$2$}
++(0,1) node[above] {$3$}
++(1,0) node[above] {$1$}
++(1,0) node[above] {$5$}
++(0,-1) node[below] {$4$}
++(-1,0) node[below] {$6$};
\path (B)++(0,2) node[left] {$(2)$};
\draw[ thick] ([c]D)-- ++(0,1)-- ++(1,0)-- ++(0,-1)--++(-1,0)
++(1,0) -- ++(0,1)-- ++(1,0)-- ++(0,-1)--++(-1,0)
++(1,0) -- ++(1,0)-- ++(0,-1)-- ++(-1,0)--++(0,1);
\fill[black!100] ([c]D) circle(0.5ex)
     ++(0,1)circle(0.5ex) ++(1,0)circle(0.5ex)
     ++(0,-1)circle(0.5ex)
     ++(1,0)circle(0.5ex)++(0,1)circle(0.5ex)
     ++(0,-2)circle(0.5ex) ++(1,0)circle(0.5ex)
     ++(0,1)circle(0.5ex);
\path (D) node[below] {$2$}
++(0,1) node[above]  {$3$}
++(1,0) node[above] {$1$}
++(1,0) node[above] {$5$}
++(-1,-1) node [below] {$6$}
++(1.2,0.3) node {$4$};
\path (D)++(0,2) node[left] {$(4)$};
\draw  [rotate around={135:(16.7,-0.7)}] (16.7,-0.7) ellipse (0.6 and 1.3);
\draw[ thick] ([c]C)-- ++(0,1)-- ++(1,0)-- ++(0,-1)--++(-1,0)
++(1,0) -- ++(0,1)-- ++(1,0)-- ++(0,-1)--++(-1,0)
 -- ++(0,-1)-- ++(1,0)-- ++(0,1);
\fill[black!100] ([c]C) circle(0.5ex)
     ++(0,1)circle(0.5ex) ++(1,0)circle(0.5ex)
     ++(0,-1)circle(0.5ex)
     ++(1,0)circle(0.5ex)++(0,1)circle(0.5ex)
     ++(0,-2)circle(0.5ex) ++(-1,0)circle(0.5ex)
     ++(0,1)circle(0.5ex);
\path (C) node[below] {$2$}
++(0,1) node[above]  {$3$}
++(1,0) node[above] {$1$}
++(1,0) node[above] {$5$}
++(-1.2,-1) node [below] {$x$}
++(1.5,0) node {$4$};
\path (C)++(0,2) node[left] {$(3)$};
\draw  [rotate around={135:(10.7,-0.7)}] (10.7,-0.7) ellipse (0.6 and 1.3);
\end{tikzpicture}
\caption{Four other subcases in Case 2}
\label{fig:case2-4sub}
\end{center}
\end{figure}
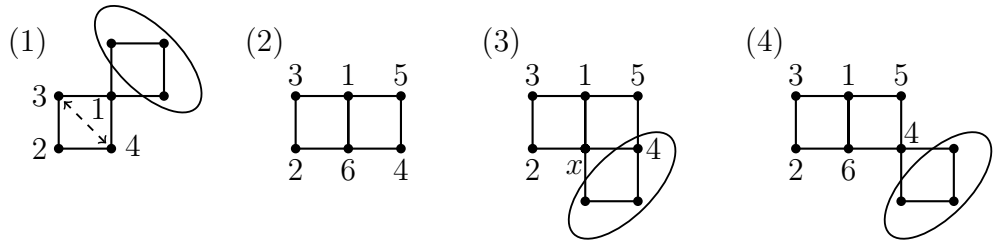

For Subcase (1), remove the nodes labeled by $2,3,4$, and by induction
hypothesis, the left graph $G'$ is represented by some word $w'$. Let $w$ be the word obtained from $w'$ by
appending $342$ to $w'$ to the left, and
 inserting $2$ directly after the leftmost $1$ in $w'$. Clearly, $w$ ensures that $3$ and $4$ are connected with $1$ and $2$ only, and $2$ is disconnected from any other node. Moreover, the third smallest
 letter $3$ occurs once at the first place of $w$.

 For Subcase (2), clearly, its $12$-representant is $w=35625124$.

 For Subcase (3), remove the nodes labeled $2$ and $3$ and let the obtained graph  be $G'$. By the induction hypothesis, assume that
 the $12$-representant of $G'$ is $w'$. Notice that  the node  labeled
 by $x$ is the only node connected to both $1$ and $4$, except the node labeled by $5$.
Then, by Corollary~\ref{coro:represntant13}, we may assume that
the $12$-representant of $G'$ is  $w'=5x1w_24w_3$
with each letter in $w_2$ and $w_3$ being larger than $5$.
Let $w=3 5 x 2 5 1 2 w_2 4 w_3$. It can be checked readily
that $w$ $12$-represents $G$ given in Subcase $(3)$.
The same construction is still valid for Subcase $(4)$, and hence,
we omit the details.

For Case 3, we swap the labels $1$ and $2$ in $G$,
then we are led to a new graph $G''$, which is just the graph
given in case $2$. It should be mentioned that
the labeling after swaping is still good, because $1$ and $2$
are indistinguishable with respect to the other elements.
Then, following from the result of case $2$, we see that
there is alway a $12$-representant for $G''$, which implies that
$G''$ is $12$-representable. By the proof of Proposition~\ref{lem:13}, we see that there is a  $12$-representant for $G$.\end{proof}

\begin{remark}\label{a-nice-open-question}
{\em We have proved that the existence of a good labeling for
a square grid graph is equivalent to the  square grid graph being $12$-representable.
Moreover, we have shown that any good labeling of a square grid graph can be
used to find a word $12$-representing the graph. It is an interesting open question if such a property holds for any other graph. Namely, is the existence of a good labeling in a graph equivalent to the graph being $12$-representable? If so, then can any good labeling be turned into a word-representant?}
\end{remark}

\begin{deff} A {\em corner node} in a $12$-representable square grid graph is a node that belongs to exactly one square, which is a non-end square and shares edges with two other squares.
\end{deff}

For example, in the leftmost graph in Figure~\ref{fig:1one}, the node labeled by 1 is a corner node, but the nodes labeled by $x$, $y$ and $z$ are not.

\begin{thm}\label{thm:corner-repeated}
Let $G$ be  a $12$-representable square grid graph. Then, there
exists  a word-representant $w$ of $G$, such that in $w$,
\begin{itemize}
\item the label of each corner node is repeated twice (for no good labeling it can appear only once in $w$);
\item between the two copies of the label of a corner node there are exactly two other letters, and
\item the label of any other (non-corner) node  appears exactly once.
\end{itemize}
\end{thm}

\begin{proof}
We wish to prove this by induction on the number of squares. The base case is given in Figure \ref{fig:corner} and its $12$-representant is
$w=351748246$.
\begin{figure}[!htbp]
\begin{center}
\begin{tikzpicture}[line width=0.7pt,scale=0.7]
\coordinate (O) at (0,0);
\draw [thick] (O)--++(0,2)--++(1,0)--++(0,-1)--++(-1,0)++(0,-1)--++(2,0)
--++(0,1)--++(-1,0)--++(0,-1);
\fill[black!100] (O) circle(0.5ex) ++(0,1) circle(0.5ex)
 ++(0,1) circle(0.5ex) ++(1,0) circle(0.5ex) ++(0,-1) circle(0.5ex)
  ++(0,-1) circle(0.5ex) ++(1,0) circle(0.5ex)
  ++(0,1)circle(0.5ex);
\path (O) node[left] {$4$}
++(0,1) node[left] {$5$}
++(0,1) node[left] {$1$}
++(1,0) node[above] {$3$}
++(0,-2) node[below] {$7$}
++(1,0) node[below] {$6$}
++(0,1) node[above] {$8$};
\path (0.7,1.3) node {$2$};
\end{tikzpicture}
\caption{The base case in the proof of Theorem \ref{thm:corner-repeated}}
\label{fig:corner}
\end{center}
\end{figure}
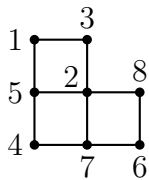

 Now, assuming that the theorem holds for square grid graphs with $n$ squares, we proceed to show that it still holds for $n+1$ squares.
This can be done by going through the cases given in Lemma~\ref{lem:extend}. In the generation of square grid graphs,
only Cases 1 and 2 will bring new corner nodes.
In Case 1, $4$ is labeled at the corner
and repeats in the $12$-representant $\overline{w}=3 5 1 w_1' 4 w_2' 2 4 w_3'$, with letters in $w_1', w_2'$ and $w_3'$ larger than $5$.
Notice that the elements in $w_2' $ are
connected with $2$, but not connected with $1,3,4,5$.
Hence, actually, there is exactly one element in $w_2'$.
This completes the proof of this case.

For the second case, $5$ is labeled at the corner
and repeats in the $12$-representant $\overline{w}=3 5  w_1' 1 5 2 w_2'  4 w_3'$, with letters in $w_1', w_2'$ and $w_3'$ larger than $5$.
Notice that the elements in $w_1' $ are
connected with $1,2,4$, but not connected with $3,5$.
Hence, there is exactly one element in $w_1'$.
\end{proof}

\section{$12$-representability of  line grid graphs} \label{sec:generalgrid}

In this section, we study the $12$-representability of
induced subgraphs of a grid graph with ``lines'' (i.e.\ edges not belonging to any square), which we call line grid graphs. Unfortunately, we cannot give a characterization in this case (we can only conjecture it; see Conjecture~\ref{conj:linegrid2}). Still, we give a number of results on $12$-representability of line grid graphs that should be useful in achieving the desired characterization.

\begin{deff}\label{def:suit}
Given a  line grid graph $G$, we call a node $v$ in $G$ {\em $k$-suitable} if
there is a way to attach an induced
path to $v$ of length $k$ that would result in a line grid graph. We refer to such induced paths as lines in what follows for brevity.
\end{deff}

\begin{prop}\label{lem:addline2}
Let $G$ be a $12$-representable square grid graph, and $G'$ is obtained from $G$ by attaching a line of length $k$, $k \geq 2$, to a $k$-suitable node in $G$. If $G'$ is $12$-representable then the line must be attached
to a node of an end-square in $G$.
\end{prop}

 \begin{figure}[!htbp]
\begin{center}
\begin{tikzpicture}[line width=0.7pt,scale=0.7]
\coordinate (O) at (0,0);
\draw ([c]O) circle (1.5);
\draw [thick] (0,1.5)--++(0,0.4);
\draw [dotted,thick] (0,1.9)--++(0,0.8);
\draw [thick] (0,2.8)--++(0,0.4);
\path (0,1.5) node[below] {$x_1$};
\path (0,2.8) node[right]{$a$}
++(0,0.4) node[right]{$b$};
\fill[black!100] (0,1.5) circle(0.5ex) ++(0,0.4) circle(0.5ex)
  ++(0,0.8)circle(0.5ex) ++(0,0.4) circle(0.5ex);
\draw [thick] (-1.07,-1.07)--++(-0.66,-0.45)--++(-0.53,0.87)--++(0.76,0.42);
\draw [thick] (1.07,-1.07)--++(0.66,-0.45)--++(0.53,0.87)--++(-0.76,0.42);
\fill[black!100] (-1.07,-1.07) circle(0.5ex)++(-0.66,-0.45)circle(0.5ex)
++(-0.53,0.87)circle(0.5ex) ++(0.76,0.42)circle(0.5ex);
\fill[black!100] (1.07,-1.07) circle(0.5ex)++(0.66,-0.45)circle(0.5ex)
++(0.53,0.87)circle(0.5ex) ++(-0.76,0.42)circle(0.5ex);
\path (0.7,-0.85) node {$x_3$} ++(0.4,0.75) node {$x_2$}
++(1.4,-0.7) node {$c$} ++(-0.55,-0.87) node {$d$};
\path (-0.7,-0.85) node {$x_4$} ++(-0.4,0.75) node {$x_5$}
++(-1.5,-0.7) node {$f$} ++(0.55,-0.87) node {$e$};
\end{tikzpicture}
\caption{A line grid graph with a line of length $k$, $k \geq 2$, attached to a $k$-suitable node not belonging to an end-square}
\label{fig:gluenonend}
\end{center}
\end{figure}
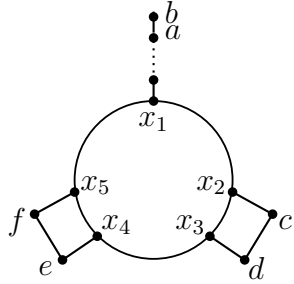

\begin{proof}
We show that a line of length $k$, $k \geq 2$, cannot be attached to a node  of a non-end-square. Indeed, assume to the contrary, that this was the case, as illustrated schematically in Figure~\ref{fig:gluenonend}. Clearly, $x_i$, $i=1,\ldots,5$, are distinct. If $1$ is the label of a node on the line, then viewing $\{x_2,x_3\}$ as a cutset,
 we have $\{c,d\}>\{e,f\}$. On the other hand, viewing $\{x_4,x_5\}$ as a cutset,
we have $\{c,d\}<\{e,f\}$, a contradiction. If $1\in \{x_2,x_3,c,d\}$, then
viewing $\{x_1\}$ as a cutset, we deduce that $\{a,b\}>\{e,f\}$. However, viewing $\{x_4,x_5\}$ as a cutset, we deduce that $\{a,b\}<\{e,f\}$,
a contradiction. Similarly, we can prove that
$1\not\in\{x_4,x_5,e,f\}$. Summarising the cases, and combining with Lemma~\ref{lem:posi-1}, we see that there is
no place for the label $1$, as desired.
\end{proof}

We believe that in Proposition~\ref{lem:addline2} the ``if then'' statement can be replaced by an ``if and only if'' statement, but we were not able to prove it.

\begin{prop}\label{thm:lineaddline}
Let $G$ be the line grid graph obtained by adding a line $A$ to a node $x$ on another line $B$  as shown in Figure~\ref{fig:lineaddline}. If $G$ is $12$-representable then the length of $A$
is at most $2$, or the smallest distance between $x$ and
an endpoint of $B$ is at most $2$.
\end{prop}

\begin{figure}[!htbp]
\begin{center}
\begin{tikzpicture}[line width=0.7pt,scale=0.7]
\coordinate (O) at (0,0);
\draw [thick] (O)--++(1,0)++(1,0)--++(2,0)++(1,0)--++(1,0);
\draw [thick] (O)++(3,0)--++(0,1)++(0,1)--++(0,1);
\draw [dotted, thick] (O)++(3,1)--++(0,1);
\draw [dotted, thick] (O)++(1,0)--++(1,0);
\draw [dotted, thick] (O)++(4,0)--++(1,0);
\fill[black!100] (O) circle(0.5ex) ++(1,0) circle(0.5ex)
 ++(1,0) circle(0.5ex) ++(1,0) circle(0.5ex) ++(1,0) circle(0.5ex)
 ++(1,0) circle(0.5ex) ++(1,0) circle(0.5ex)++(-3,1) circle(0.5ex)
 ++(0,1) circle(0.5ex)++(0,1) circle(0.5ex);
\path (O)  node[below] {$i$}
 ++(2,0) node[below] {$p$}
 ++(1,0) node[below] {$x$}
 ++(1,0) node[below] {$q$}
 ++(-1,1) node[right] {$r$}
 ++(3,-1)node[below] {$j$} ;
\draw [->, very thick] (3.9,2.5) to (3.2,2.5);
\path (4.2,2.5) node  {$A$};
\draw [->, very thick] (5.5,0.9) to (5.5,0.2);
\path (5.5,1.2) node  {$B$};
\end{tikzpicture}
\caption{A line grid graph with a line of length $k$, $k \geq 2$, not attached to a node of an end-square}
\label{fig:lineaddline}
\end{center}
\end{figure}
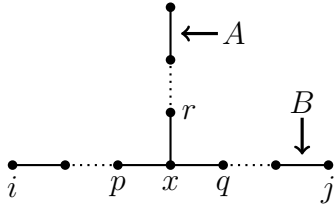

\begin{proof}
Assume that  the part of line $B$ between $i$ and $x$ (resp., $x$ and $j$)
 is called $C$ (resp., $D$). We wish to prove that if all of the lengths of $A$, $C$ and $D$ are larger than $2$, then $G$ is not $12$-representable.
We prove this by showing that there is no good labeling for $G$
in this case. If $1$ is the label of a node on $A$, then
choosing $p$ as a cutset, we see that the nodes on $C$, except for $p$ and $x$, are larger than those on $D$.  Choosing $q$ as a cutset, we see that the nodes in $D$, except for $q$ and $x$, are larger than those on $C$, which is a
contradiction. Hence, $1$ cannot be the label of a node on $A$. By a similar analysis, we can also show that $1$ cannot be the label of a node on $C$ or $D$, as desired.\end{proof}

Note that the graph in Figure~\ref{fig:lineaddline} is a subdivision of the claw $K_{1,3}$. Denote a graph of this form with three branches of length $s,t,p$ by $B(s,t,p)$ and observe that Proposition~\ref{thm:lineaddline} proves that $B(3,3,3)$, the 3-subdivision of $K_{1,3}$ presented in Figure~\ref{fig:B333}, is not 12-representable. Once again, we believe that in Proposition~\ref{thm:lineaddline} the ``if then'' statement can be replaced by an ``if and only if'' statement, but we were not able to prove it.

\begin{figure}[!htbp]
\begin{center}
\begin{tikzpicture}[line width=0.7pt,scale=0.7]
\coordinate (O) at (0,0);
\draw [thick] (O)--++(1,0)++(1,0)--++(2,0)++(1,0)--++(1,0);
\draw [thick] (O)++(3,0)--++(0,1)++(0,1)--++(0,1);
\draw [thick] (O)++(3,1)--++(0,1);
\draw [thick] (O)++(1,0)--++(1,0);
\draw [thick] (O)++(4,0)--++(1,0);
\fill[black!100] (O) circle(0.5ex) ++(1,0) circle(0.5ex)
 ++(1,0) circle(0.5ex) ++(1,0) circle(0.5ex) ++(1,0) circle(0.5ex)
 ++(1,0) circle(0.5ex) ++(1,0) circle(0.5ex)++(-3,1) circle(0.5ex)
 ++(0,1) circle(0.5ex)++(0,1) circle(0.5ex);
\end{tikzpicture}
\caption{The non-12-representable graph $B(3,3,3)$}
\label{fig:B333}
\end{center}
\end{figure}
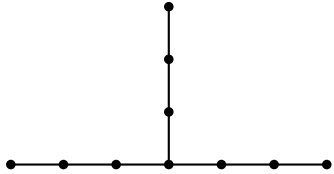

\begin{thm}\label{thm:addoneline}
Given a $12$-representable line grid graph $G$,
 $G'$  is obtained by gluing a line of length $1$ to a node of $G$ which is not a corner node and is $1$-suitable.  Then, $G'$ is $12$-representable.
\end{thm}

\begin{proof}
First, we claim that each node in $G$ is either a local maximum or a local minimum,  and moreover,  a local maximum must connect to a local minimum. Otherwise, there will occur an induced subgraph equal to
$I_3$, which contradicts to the fact that $G$ is $12$-representable.

By Theorem \ref{thm:corner-repeated}, we may find a labeling of $G$ with $w$ being its $12$-representant, where each non-corner label occurs
only once. Assume that the node we glue to is $i$.
Now, we consider two cases as follows.

\noindent
{\bf Case 1.} $i$ is a local maximum. Labeling the node new added $i$ and increasing each label in $G$
   not less than $i$ by $1$, we obtain a labeling of $G'$.
   Let $w'$ be the word obtained from $w$ by increasing each letter in
   $w$ larger than $i$ by $1$, changing $i$ to $(i+1)i$ and insering
   an $i$ at the end of the word.

   We claim that $w'$  $12$-represents $G'$.  Since $i$ is a local maximum of $G$,  all letters larger than $i+1$ will occur after $(i+1)i$ in $w'$. This implies that $i$ is disconnected with all letters larger than $i+1$. The $i$ at the end of $w'$ kills all smaller letters. Lastly, $i$ and $i+1$ are connected. The claim is
   verified.

\noindent
{\bf Case 2.} $i$ is a local minimum. Labeling the node new added $i+1$ and increasing each label in $G$
   larger than $i$ by $1$, we obtain a labeling of $G'$.
   Let $w'$ be the word obtained from $w$ by increasing each letter in
   $w$ larger than $i$ by $1$, then changing $i$ to $(i+1)i$ and lastly inserting an $i+1$ at the begining of the word.
   By a similar analysis as given in case $1$, we may see that $w'$  $12$-represents $G'$.

Summarizing, in each case, we find a $12$-representant for $G'$.
Hence, $G'$ is $12$-representable and we complete the proof.
\end{proof}

\begin{prop}\label{lem:corneradd1}
Assume that $G'$ is obtained by gluing  a line of length $1$ to a corner node  which is $1$-suitable in $G$ , then $G'$ is not $12$-representable.
\end{prop}

\begin{proof}
The statement follows from non-$12$-representability of the graph $X$ in Figure~\ref{fig:corneradd1}.
\end{proof}

%
%

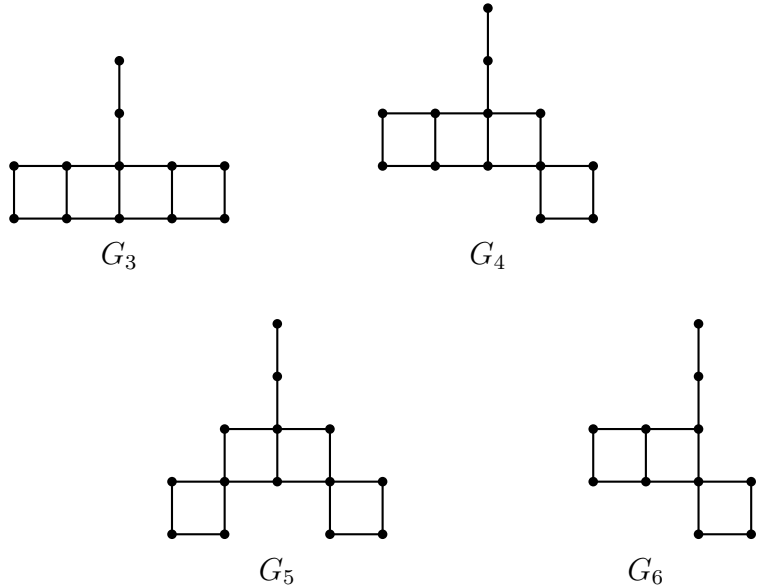
\begin{figure}[!htbp]
\begin{center}
\begin{tikzpicture}[line width=0.7pt,scale=0.7]
\coordinate (O) at (0,0);
\coordinate (B) at (3,-6);
\draw [thick] (O)--++(4,0)--++(0,1)--++(-4,0)--++(0,-1)
++(1,0)--++(0,1)++(2,0)--++(0,-1)++(-1,0)--++(0,3);

\draw [thick] (O)++(7,1)--++(4,0)--++(0,-1)--++(-1,0)--++(0,2)
--++(-3,0)--++(0,-1)++(1,0)--++(0,1)++(1,-1)--++(0,3);


\fill[black!100] (O) circle(0.5ex) ++(1,0) circle(0.5ex)
 ++(1,0) circle(0.5ex)++(1,0) circle(0.5ex)++(1,0) circle(0.5ex)
 ++(0,1) circle(0.5ex)++(-1,0) circle(0.5ex)++(-1,0) circle(0.5ex)
 ++(-1,0) circle(0.5ex)++(-1,0) circle(0.5ex)++(2,1) circle(0.5ex)
 ++(0,1) circle(0.5ex);

\fill[black!100] (O)++(7,1) circle(0.5ex) ++(1,0) circle(0.5ex)
 ++(1,0) circle(0.5ex)++(1,0) circle(0.5ex)++(1,0) circle(0.5ex)
 ++(0,-1) circle(0.5ex)++(-1,0) circle(0.5ex)++(0,2) circle(0.5ex)
 ++(-1,0) circle(0.5ex)++(-1,0) circle(0.5ex)++(-1,0) circle(0.5ex)++(2,1) circle(0.5ex)
 ++(0,1) circle(0.5ex);


 \path (O)++(2,-0.7) node {$G_3$};
 \path (O)++(9,-0.7) node {$G_4$};

\draw [thick] (B)--++(1,0)--++(0,2)--++(2,0)--++(0,-2)
--++(1,0)--++(0,1)--++(-4,0)--++(0,-1)++(2,1)--++(0,3);

\draw [thick] (B)++(8,1)--++(2,0)--++(0,3)++(0,-2)--++(-2,0)--++(0,-1)
++(1,0)--++(0,1)++(1,-1)--++(1,0)--++(0,-1)--++(-1,0)--++(0,1);

\fill[black!100] (B) circle(0.5ex) ++(1,0) circle(0.5ex)
 ++(0,1) circle(0.5ex)++(0,1) circle(0.5ex)
 ++(1,0) circle(0.5ex)++(1,0) circle(0.5ex)++(0,-1) circle(0.5ex)
 ++(-1,0) circle(0.5ex)++(1,-1) circle(0.5ex)++(1,0) circle(0.5ex)++(0,1) circle(0.5ex)
 ++(-4,0) circle(0.5ex)++(2,2) circle(0.5ex)++(0,1) circle(0.5ex);

\fill[black!100] (B)++(8,1) circle(0.5ex) ++(1,0) circle(0.5ex)
 ++(1,0) circle(0.5ex)++(0,1) circle(0.5ex)
 ++(-1,0) circle(0.5ex)++(-1,0) circle(0.5ex)++(2,1) circle(0.5ex)
 ++(0,1) circle(0.5ex)++(0,-4) circle(0.5ex)++(1,0) circle(0.5ex)
 ++(0,1) circle(0.5ex);

 \path (B)++(2,-0.7) node {$G_5$};
 \path (B)++(9,-0.7) node {$G_6$};

 \end{tikzpicture}
\caption{The non-$12$-representable graphs  $G_i$, $i \in \{3,4,5,6\}$; non-$12$-representability of these graphs follows from Proposition~\ref{lem:addline2}}
\label{fig:G_3to7}
\end{center}
\end{figure}

We end up this paper with the following conjecture.

\begin{conj}\label{conj:linegrid2}
A line grid graph is $12$-representable if and only if it is $P$-avoiding,
where $P:=F \cup \{B(3,3,3),G_3,G_4,G_5,G_6\}$ with $B(3,3,3)$ and $G_i, i\in \{3,4,5,6\}$, given in Figures~\ref{fig:B333} and~\ref{fig:G_3to7}, respectively.
 \end{conj}

\section*{Acknowledgement}
The authors are grateful to the referee for providing many useful comments that helped to improve the presentation of the paper. In particular, Conjecture~\ref{conj:linegrid2} was modified based on referee's comments. The first author was supported by Natural Science Foundation Project of Tianjin Municipal Education Committee (No.~2017KJ243,~No.~2018KJ193)
and the National Natural Science Foundation of China (No.~11701420).

\end{document}